\documentclass[12pt,a4paper]{amsart}
\usepackage[a4paper, left=28mm, right=28mm, top=28mm, bottom=34mm]{geometry}
\usepackage{latexsym}
\usepackage{amsmath}
\usepackage{amssymb}
\usepackage{amsthm}
\usepackage{amscd}
\usepackage{mathrsfs}
\usepackage{mathtools}
\usepackage[all]{xy}
\usepackage{graphicx}
\usepackage{comment}
\usepackage{arydshln}
\usepackage{stmaryrd}
\usepackage{url}

\newtheorem{theorem}{Theorem}[section]
\newtheorem{lemma}[theorem]{Lemma}
\newtheorem{corollary}[theorem]{Corollary}
\newtheorem{proposition}[theorem]{Proposition}

\newtheorem{example}[theorem]{Example}
\theoremstyle{definition}

\newtheorem{assumption}[theorem]{Assumption}
\newtheorem{remark}[theorem]{Remark}
\newtheorem{definition}[theorem]{Definition}

\def\A{{\mathbb A}}
\def\F{{\mathbb F}}
\def\Q{{\mathbb Q}}
\def\Z{{\mathbb Z}}

\def\oU{\overline{U}}
\def\a{\alpha}

\theoremstyle{remark}

\makeatletter

\@addtoreset{equation}{section}
\makeatother

\makeatletter

\@addtoreset{equation}{section}
\makeatother

\makeatletter
\newcommand{\subsubsubsection}{\@startsection{paragraph}{4}{\z@}%
 {1.0\Cvs \@plus.5\Cdp \@minus.2\Cdp}%
 {.1\Cvs \@plus.3\Cdp}%
 {\reset@font\sffamily\normalsize}
 }
\makeatother
\DeclareMathOperator{\Spec}{Spec}

\DeclareMathOperator{\Gal}{Gal}
\DeclareMathOperator{\id}{id}

\DeclareMathOperator{\Hom}{Hom}
\DeclareMathOperator{\Ker}{Ker}

\DeclareMathOperator{\Tr}{Tr}

\usepackage{bm}

\usepackage[usenames,dvipsnames]{color}

\begin{document}

\title[$L$-polynomials of van der Geer--van der Vlugt curves]{The $L$-polynomials of van der Geer--van der Vlugt curves in characteristic $2$}

\author{Tetsushi Ito}
\address{
Department of Mathematics, Faculty of Science, Kyoto University
Kyoto, 606--8502, Japan}
\email{tetsushi@math.kyoto-u.ac.jp}

\author{Daichi Takeuchi}
\address{Department of Mathematics,
Institute of Science Tokyo,
2-12-1 Ookayama, Meguro-ku, Tokyo, 152-8551, Japan
}
\email{daichi.takeuchi4@gmail.com}

\author{Takahiro Tsushima}
\address{
Keio University School of Medicine,
4-1-1 Hiyoshi, Kohoku-ku,
Yokohama, 223-8521, Japan}
\email{tsushima@keio.jp}

\date{\today}



\begin{abstract}
The van der Geer--van der Vlugt curves form a class of Artin--Schreier coverings of the projective line over finite fields.
We provide an explicit formula for their $L$-polynomials in characteristic $2$, expressed in terms of characters of maximal abelian subgroups of the associated Heisenberg groups.
For this purpose, we develop new methods specific to characteristic $2$ that exploit the structure of the Heisenberg groups
and the geometry of Lang torsors for $W_2$.
As an application, we construct examples of curves in this family attaining the Hasse--Weil bound.
\end{abstract}

\maketitle

\section{Introduction}
\label{Introduction}

Let $p_0$ be a prime number, $p$ a power of $p_0$, and $q$ a power of $p$.
A polynomial of the form $R(x) = \sum_{i=0}^{e} a_i x^{p^i} \in \F_q[x]$ \ $(a_e \neq 0)$
is called an $\F_p$-linearized polynomial.
Let $C_R$ be the affine curve over $\F_q$ defined by the equation
$y^p - y = x R(x)$.
Let $\overline{C}_R$ be the smooth compactification of $C_R$.
The curves $C_R$ and $\overline{C}_R$ are called the \emph{van der Geer--van der Vlugt curves}
associated with $R(x)$.
When $p = p_0$, they were studied by van der Geer--van der Vlugt \cite{GV}.
The van der Geer--van der Vlugt curves have interesting geometric and arithmetic properties; see, for example, \cite{BHMSSV, C, ITT, OS, Tatematsu}.
Moreover, these curves are generally expected to be related to the reductions of certain affinoids in the Lubin--Tate space; examples of this are given in \cite{IT, We}. Understanding explicitly the Frobenius eigenvalues on their \'etale cohomology is expected to shed light on the explicit local Langlands correspondence.

In \cite{TT}, the second and third authors obtained an explicit formula for the $L$-polynomial of $\overline{C}_R$. Let us briefly recall the results. For details, we refer to \cite{TT} or Section \ref{Section:Review} of this paper.
Assume $e \geq 1$, which implies that the genus of $\overline{C}_R$ is positive.
Let $\F$ be an algebraic closure of $\F_q$.
Let $H_R$ be the Heisenberg group acting on $C_{R,\F} \coloneqq C_R \otimes_{\F_q} \F$.
The underlying set of $H_R$ is naturally identified with a subset of $\F^2$.
The center $Z(H_R)$ is naturally identified with $\F_p$, and the quotient $H_R/Z(H_R)$ is abelian.
Fix a maximal abelian subgroup $A\subset H_R$. Throughout the paper, we assume that $A \subset \F_q^2$. Then the action of $A$ on $C_{R,\F}$ can be defined over $\F_q$,
and the quotient $C_R/A$ is isomorphic to the affine line $\A^1_{\F_q}$ over $\F_q$.
Therefore, we obtain a finite \'etale Galois covering $C_R\to \A^1_{\F_q}$ with Galois group $A$. 

Let $\ell$ be a prime number different from $p_0$.
For a finite abelian group $G$, we write $G^\vee$ for its group of $\overline{\Q}_{\ell}^{\times}$-valued characters 
$ \Hom_{\Z}(G,\overline{\Q}_{\ell}^{\times})$.
For a nontrivial character $\psi \in \F_p^{\vee}
\setminus \{1\}$, 
let $A_{\psi}^{\vee}$ be the set of characters $A\to \overline{\Q}_{\ell}^{\times}$
whose restrictions to $Z(H_R) \simeq \F_p$ are $\psi$. Using the Galois covering $C_R\to \A^1_{\F_q}$, we attach to each $\xi\in A_\psi^\vee$ a smooth $\overline{\Q}_{\ell}$-sheaf 
$\mathcal{Q}_{\xi}$ on  $\A^1_{\F_q}$.
It is proved in \cite[Lemma~3.4]{TT} that the $\ell$-adic cohomology group
$H_{\rm c}^1(\A^1_{\F},\mathcal{Q}_{\xi})$
is one-dimensional. 
Let $\Gal(\F/\F_q)$ denote the Galois group of the extension $\F/\F_q$. 
Let $\tau_{R,\xi,q} \in \overline{\Q}_{\ell}^{\times}$
denote the eigenvalue of
the geometric Frobenius element $\mathrm{Fr}_{q} \in \Gal(\F/\F_q)$
acting on $H_{\rm c}^1(\A^1_{\F},\mathcal{Q}_{\xi})$.
The following formula is obtained in \cite[Theorem~1.1]{TT}:
\[
L_{\overline{C}_R/\F_q}(T) \coloneqq \det(1-\mathrm{Fr}_qT \mid  
H^1(\overline{C}_{R,\F},
\overline{\Q}_{\ell}))
= \prod_{\psi \in \F_p^{\vee}
\setminus \{1\}} \prod_{\xi \in A_{\psi}^{\vee}}
(1-\tau_{R,\xi,q}T). 
\]
Therefore, a detailed analysis of the sheaf $\mathcal{Q}_\xi$ yields an explicit formula for the eigenvalue $\tau_{R,\xi,q}$, and consequently, for the $L$-polynomial. 

When $p_0$ is odd, we construct in \cite[A.2]{TT} an intermediate curve $C$, defined by the equation $y^p-y=c_Ax^2$ with $c_A\in\F_q^{\times}$ being explicitly computable, that lies between the covering $C_R\to \A^1_{\F_q}=C_R/A$, i.e., we have a factorization $C_R\to C\to \A^1_{\F_q}$. Using this factorization, it is shown in loc.~cit. that $\mathcal{Q}_\xi$ is isomorphic to an Artin--Schreier sheaf of the form $\mathcal{L}_\psi(c_Ax^2+b_\xi x)$, where $b_\xi\in\F_q$ is a constant determined by $\xi$. As a consequence of the Grothendieck--Lefschetz trace formula, the eigenvalue $\tau_{R,\xi,q}$ is expressed in terms of a quadratic Gauss sum and a value of the additive character $\psi$ in \cite[Corollary~A.8]{TT}. 
However, the method described above applies only in the case of odd characteristic, because the construction of $C$ relies on the fact that $A$ is annihilated by $p_0$, which is false when $p_0=2$. 
The aim of this paper is to develop a similar method that is applicable to the case of characteristic $2$, so that we obtain a formula for $\tau_{R,\xi,q}$ in this remaining case.

Let us explain our results in more detail. From now on, we assume that $p_0=2$. In this case, we show that $C_R\to \A^1_{\F_q}$ factors as 
\begin{equation}\label{CRSA}
C_R\to C_{S}\to \A^1_{\F_q},
\end{equation}
where $C_S$ is the van der Geer--van der Vlugt curve associated with $S(x)=x^p+s_0x$. Moreover, we prove that the covering $C_S\to\A^1_{\F_q}$ is Galois with Galois group $W_2(\F_p)$, where $W_2$ denotes the ring of Witt vectors of length $2$. The appearance of $W_2(\F_p)$ reflects the fact that $A$ has nontrivial $4$-torsion. 
We study the $\ell$-adic cohomology of $C_S$
using a geometric construction inspired by the work of Deligne--Lusztig.

Using the $W_2(\F_p)$-covering $C_S\to\A^1_{\F_q}$, we compute $\mathcal{Q}_\xi$ and $\tau_{R,\xi,q}$ as follows. First, we note that \eqref{CRSA} induces a surjection $A\to W_2(\F_p)$. Since $A$ is annihilated by $4$ (Lemma~\ref{basic} (4)) and $W_2(\F_p)$ is $\Z/4\Z$-free as in Lemma~\ref{free}, this map splits, and a splitting induces an isomorphism 
\begin{equation}\label{introsplit}
    A\simeq W_2(\F_p)\times \oU. 
\end{equation}
Here $\oU$ denotes the kernel of $A\to W_2(\F_p)$, which is naturally an $\F_p$-vector space of dimension $e-1$. 

We fix such an isomorphism in \eqref{introsplit}. 
Then the character $\xi\in A^\vee_\psi$ can be written as $\xi = \xi_{W_2} \boxtimes \xi_{\oU}$
for some characters $\xi_{W_2} \in W_2(\mathbb{F}_p)^{\vee}$ and  $\xi_{\oU} \in \oU^{\vee}$. 
Let 
$\mathcal{L}_{\xi_{W_2}}(z,w)$ denote the invertible 
$\overline{\Q}_{\ell}$-sheaf on
$W_{2,\F_p}$
defined by the Lang torsor 
$W_{2,\F_p} \to W_{2,\F_p},\ (x,y) \mapsto 
(x^p,y^p)-(x,y)$ 
and the character $\xi_{W_2} \in W_2(\F_p)$,
where $W_{2,\F_p} \simeq \Spec \F_p[z,w]$
 denotes the scheme of Witt vectors of length $2$.
For a morphism $f \times g \colon \A_{\F_q}^1
\to W_{2,\F_q}$, we write 
 $\mathcal{L}_{\xi_{W_2}}(f(s), g(s))$
for its pull-back  by $f\times g$. 
Using this notation,  we can now state our result on $ \mathcal{Q}_{\xi}$, which is proved in Proposition~\ref{qpp}, as follows: 
 \[
 \mathcal{Q}_{\xi} \simeq \mathcal{L}_{\xi_{W_2}}(s, \a s^2 + \beta_{\xi} s). 
 \]
Here, $\alpha$ and $\beta_{\xi}$ are certain elements of $\F_q$ whose definitions are given after Theorem~\ref{MainTheorem}.

This isomorphism together with the computation of the ranks of the cohomology groups of $\mathcal{Q}_\xi$ (see Lemma~\ref{rkQ}) gives us the following result on $\tau_{R,\xi,q}$. 
Fix a faithful character
$\xi_2 \colon W_2(\mathbb{F}_2)  \to \overline{\Q}_{\ell}^{\times}$ and put  $\sqrt{-1} \coloneqq \xi_2(1,0) \in \overline{\Q}_{\ell}^{\times}$. Since $W_2(\F_2)$ is isomorphic to $\Z/4\Z$, the element $\sqrt{-1}$ is a primitive $4$-th root of unity. Let $\Tr_{q/2} \colon W_2(\F_q) \to W_2(\F_2)$ denote the trace map.
We set $\xi_q \coloneqq \xi_2 \circ \Tr_{q/2} \in W_2(\F_q)^{\vee}$.

The first main theorem of this paper gives an explicit formula for the Frobenius eigenvalue $\tau_{R,\xi,q}$ associated with a van der Geer--van der Vlugt curve in terms of additive characters of the Witt vector group $W_2(\F_q)$.

\begin{theorem}[see Theorem~\ref{qppc}]
\label{MainTheorem}
Let $a_{W_2},b_{W_2}$ be elements of $\F_p$ defined below. 
Define 
\[ c_{R,\xi} \coloneqq (\alpha + (b_{W_2}/a_{W_2}^2))^{1/2} + a_{W_2} \beta_{\xi} \in \F_q. \] 
Then we have
\[
  \tau_{R,\xi,q} = \xi_q(c_{R,\xi}, 0)^{-1}
  \cdot (-1-\sqrt{-1})^{[\F_q : \F_2]}, 
\]
where we regard $(c_{R,\xi},0) \in W_2(\F_q)$.

\end{theorem}
The definitions of $\alpha$, $\beta_{\xi} \in \F_q$ and $a_{W_2}, b_{W_2} \in \F_p$ are as follows: 
\begin{itemize}
\item Let $\alpha \in \F_q$ satisfy
$\alpha^{p^{-1}} + \alpha = s_0 + 1$.
\item 
Let $\overline{A} \coloneqq A/Z(H_R)$, which naturally has a structure of an 
$\F_p$-vector space of dimension $e$. The vector space $\overline{A}$ is naturally 
regarded as a subspace of $\F_q$.
By Lemma~\ref{aF=Fa}, there exist separable $\F_p$-linearized polynomials $F_A(x), a(x) \in \F_q[x]$ such that 
\begin{gather}\label{oaa}
\begin{aligned}
\overline{A}= \{x \in \F \mid F_A(x)=0\}, \\ 
x^q + x = a \circ F_A(x) = F_A \circ a(x).
\end{aligned}
\end{gather}
Furthermore, $a(x)$ defines a homomorphism 
$\F_q \to \overline{A},\ x \mapsto a(x)$. 
Via the isomorphism \eqref{introsplit}, we define $r$ to be the projection
onto the second factor,
$r \colon \overline{A}\to\oU$. 
Let 
$\beta_{\xi} \in \mathbb{F}_q$ be the unique element such that 
$\xi_{\oU} \circ r \circ a(x)=\psi \circ  \Tr_{\F_q/\F_p}(\beta_{\xi} x)$ 
for all $x \in \mathbb{F}_q$.

\item
There exists a unique element $(a_{W_2}, b_{W_2}) \in W_2(\F_p)^{\times}$
such that
\[
  \xi_{W_2}(x,y) = \xi_p((a_{W_2},b_{W_2}) \cdot (x,y))
\]
for every $(x,y) \in W_2(\F_p)$.
Here, 
``$(a_{W_2},b_{W_2}) \cdot (x,y)$'' denotes the multiplication in the ring $W_2(\F_p)$.
\end{itemize}

We also study a relation between Frobenius eigenvalues of $\overline{C}_R$ and those of a twist of it.
Here, a \emph{twist} of $C_R$ means the van der Geer--van der Vlugt curve defined by the polynomial $R(x)+ax$ for some $a\in \F_q$. 

To state our result, we introduce some notation. Let $F_A(x)$ be an $\F_p$-linearized polynomial as in \eqref{oaa}. We may assume that $F_A$ is of the form $f^p+f$ for some polynomial $f$ (see Proposition~\ref{choice}). We write $F_A(x)=\sum_{i=0}^e b_i x^{p^i}$ with $b_e \neq 0$.
For $t \in \F_q$, we define $F^\ast(t)
\coloneqq \sum_{i=0}^e (b_i t)^{p^{-i}}$ and $R_t(x) \coloneqq  R(x)+F^\ast(t)^2 x$. Let 
$\pi_t \colon H_{R_t} \to V_{R_t}=V_R,\ (a,b) \mapsto a$ and
$A_t\coloneqq \pi_t^{-1}(\overline{A})$. 
We have an isomorphism 
$A_t \simeq W_2(\F_p) \times \oU$ similarly as \eqref{introsplit}. Then we can regard 
the character $\xi_{W_2} \boxtimes \xi_{\oU}$ as
 a character of $A_t$, which we denote by $\xi_t$. 

The following result is the second main theorem of this paper,
which asserts that a Frobenius eigenvalue of the twist $\overline{C}_{R_t}$
is the product of a Frobenius eigenvalue of $\overline{C}_{R}$
times a $4$-th root of unity, which can be calculated explicitly.

\begin{theorem}[see Theorem~\ref{qppt}]\label{MT2}
We have
\[
  \tau_{R_t,\xi_t,q} = \xi_q(t,t^2+c_{R,\xi}t )
  \cdot  \tau_{R,\xi,q}. 
\]
\end{theorem}

As an application of Theorem~\ref{MT2}, we construct examples of $\F_{q^2}$-maximal curves by explicitly twisting $\F_{q^2}$-minimal ones. For the precise definitions of maximal and minimal curves, we refer the reader to Subsection \ref{Subsection:maximalcurve}.

The following theorem illustrates a concrete instance of this phenomenon.

\begin{theorem}
[see Theorem~\ref{main} (1)]\label{MT3}
Let $t \in \F_q$ be such that $\Tr_{q/2}(t)=1$.
Assume that $A \subset \F_q^2$ and 
$\overline{C}_R$ is $\F_{q^2}$-minimal.  
Then $\overline{C}_{R_t}$ is $\F_{q^2}$-maximal. 
\end{theorem}

Using the results stated above,
we give explicit examples of maximal curves in Proposition~\ref{2pp}, Corollaries~\ref{abc}, \ref{abcd}, and~\ref{maximalexample}.

\begin{remark}
Our approach in this paper is geometric in nature and based on studying the quotients of van der Geer--van der Vlugt curves
by abelian subgroups of the Heisenberg groups.
Using \emph{quadratic $\ell$-adic sheaves}
introduced by the second author \cite{Ta},
it is possible to study the Frobenius eigenvalues
by cohomological methods.
We plan to pursue the cohomological techniques in a subsequent paper.
\end{remark}
In \cite{BP}, Blache and Pierre give a factorization of the $L$-functions of the van der Geer--van der Vlugt curves (which they refer to as  quadratic Artin--Schreier curves) using exponential sums. Our approach differs from theirs and 
provides a geometric and representation-theoretic perspective, based on $\ell$-adic cohomology.

In \cite[Section~9.2]{KT}, 
 Katz and Tiep study in detail certain local systems constructed  using  Witt vectors of length two. In a special case, their local systems are closely related to those  considered in this article. See Remark~\ref{relation with KT} for details.

\subsubsection*{The organization of this paper}

This paper is organized as follows.
In Section \ref{Section:Review},
we summarize basic properties of van der Geer--van der Vlugt curves and their $L$-polynomials following \cite{TT}.
In Section \ref{Section:Geometry}, 
we study the geometry and cohomology of the curve $C_S$ defined by the equation
$y^p + y = x^{p+1} + a_0 x^2$
in some detail.
In Section \ref{Section:Quotients},
we study quotient maps between van der Geer--van der Vlugt curves. In particular, we prove that for any 
$\F_p$-linearized polynomial $R(x)$, there exists a finite \'etale map $C_R\to C_S$, where $C_S$ is the curve associated with $S(x)=x^p+s_0x$ for some $s_0 \in \F_q$. 
In Section \ref{Section:MainTheorem}, we prove Theorem~\ref{MainTheorem}
using results from Sections \ref{Section:Geometry} and \ref{Section:Quotients}.
In Section \ref{Section:maximalcurve},
we give applications of our results to
the construction of maximal curves.
In Section \ref{Section:Twists},
we prove Theorems~\ref{MT2} and~\ref{MT3}.
Finally, in Section \ref{Section:supersingular},
we show how supersingular elliptic curves appear as quotients of the curve $C_S$ studied in Section \ref{Section:Geometry}.

\subsubsection*{Notation}

We use the same notation as in \cite{ITT, TT}.

Let $p_0$ be a prime number, 
let $p$ be a power of $p_0$, and let $q$ be a power of $p$. We write $p=p_0^{f_0}$. 
Fix an algebraic closure $\mathbb{F}$ of $\F_{p_0}$. We regard finite extensions of $\F_{p_0}$ as subfields of $\mathbb{F}$.
Hence we have a tower of fields
$\F_{p_0} \subset \F_p \subset \F_q \subset \F$.
The trace map from $\F_q$ to $\F_p$ is denoted by 
$\Tr_{q/p} \colon \mathbb{F}_{q} \to \mathbb{F}_p$.

Throughout the paper, we fix a prime number $\ell$ different from $p_0$.
For a finite abelian group $G$, we write $G^\vee$ for the group of $\overline{\Q}_{\ell}^{\times}$-valued characters 
$ \Hom_{\Z}(G,\overline{\Q}_{\ell}^{\times})$.

For a smooth affine curve $C$ over $\F_q$, let 
$\overline{C}$ denote the smooth compactification of $C$. 

Let $X$ be a scheme over $\F_q$. 
Its base change to $\F$ is denoted by 
$X_{\F} \coloneqq X \otimes_{\F_q} \F$.

For two polynomials 
$f_1(x), f_2(x) \in \F_q[x]$, we write $f_1(x) \sim f_2(x)$ if there exists a polynomial $d(x) \in \F_q[x]$
such that $d^p-d=f_1-f_2$.

\section{Review of van der Geer--van der Vlugt curves}
\label{Section:Review}

Here we summarize basic properties of van der Geer--van der Vlugt curves following \cite{GV, TT}. 
Except for 
Lemma~\ref{basic} (4), Proposition~\ref{14} (3), and 
Lemma~\ref{Stone} (3), the results in this section are valid in all characteristics.

\subsection{The van der Geer--van der Vlugt curves associated with $\F_p$-linearized polynomials}

A polynomial $f(x)\in \F_q[x]$ is called \emph{$\F_p$-linearized} if it is additive,
that is, $f(x+y)=f(x)+f(y)$, and satisfies $f(ax)=af(x)$ for all $a\in\F_p$. 
Such a polynomial can be written as $f(x) = \sum_{i=0}^{e} c_i x^{p^i}$, with $c_i \in \F_q$.
 Let
 \[ \mathscr{A} \coloneqq \bigg\{ \sum_{i=0}^{e} c_i x^{p^i}
 \in \F_q[x] \ \bigg| \ e \in \mathbb{Z}_{\geq 0} \bigg\} \]
 denote the set of $\F_p$-linearized 
 polynomials with coefficients in $\F_q$. 
For $f(x),g(x) \in \mathscr{A}$, define 
$
(f \circ g)(x) \coloneqq f(g(x)) \in \mathscr{A}. 
$
Together with the usual addition, this gives $\mathscr{A}$ the structure of a ring under composition.

For $f(x) \in \mathscr{A}$, we define 
$\Ker f \coloneqq \{x \in \mathbb{F} \mid 
f(x)=0\}$,
which is an $\mathbb{F}_p$-vector space. If $f(x)$ has degree $p^e$ and is separable, then it has dimension $e$.

Let $R(x)\in\mathscr{A}$ be an $\F_p$-linearized polynomial of degree $p^e$ with $e\geq1$.
We define 
the smooth affine curve $C_R$ over $\F_q$ by the equation 
$
y^p - y = x R(x).
$
Its smooth compactification $\overline{C}_R$ is a smooth projective curve of genus $p^e(p-1)/2$;
see \cite[Proposition~6.4.1 (e)]{St}.
The curves $C_R$ and $\overline{C}_R$ are called the \emph{van der Geer--van der Vlugt curves}
over $\F_q$ associated with $R(x)$.

Our main aim in this section is to introduce a finite \'etale Galois covering
$F_A \circ \phi \colon C_R \to \A_{\mathbb{F}_q}^1$
and smooth $\overline{\Q}_{\ell}$-sheaves $\mathcal{Q}_{\xi}$ on $\A_{\mathbb{F}_q}^1$,
and decompose the $\ell$-adic \'etale cohomology of $C_{R,\F} \coloneqq C_R \otimes_{\F_q} \F$;
see Proposition~\ref{14}.

\subsection{Heisenberg groups and their maximal abelian subgroups}

Let
$R(x) = \sum_{i=0}^e a_i x^{p^i} \in \mathscr{A}$
be an $\F_p$-linearized polynomial with $e\geq1$ and $a_e\neq0$.
The curve $C_R$ carries an action by a certain Heisenberg group $H_R$. We recall this fact briefly.

We introduce the polynomials
$E_R(x) \in \F_q[x]$
and $f_R(x,y) \in \F_q[x,y]$ defined by
\begin{align} \label{-aa}
E_R(x) &\coloneqq R(x)^{p^e}+\sum_{i=0}^e (a_ix)^{p^{e-i}}, \\ \label{-a} 
f_R(x,y) &\coloneqq -\sum_{i=0}^{e-1}\left(\sum_{j=0}^{e-i-1}(a_i x^{p^i} y)^{p^j}
+(x  R(y))^{p^i}\right).
\end{align}

The polynomial $f_R(x,y)$ is the unique polynomial 
satisfying $f_R(0,0)=0$ and 
\begin{equation}\label{a}
f_R(x,y)^p-f_R(x,y)=-x^{p^e} E_R(y)+xR(y)+y R(x). 
\end{equation}
It follows that the polynomial 
\begin{equation}\label{om}
\omega_R(x,y)\coloneqq f_R(x,y)-f_R(y,x)
\end{equation}
satisfies $\omega_R(0,0)=0$ and 
\begin{equation}\label{omega}
\omega_R(x,y)^p-\omega_R(x,y)=y^{p^e} E_R(x)-x^{p^e}E_R(y). 
\end{equation}

We set 
\[
V_R \coloneqq \Ker E_R, 
\]
which is an $\F_p$-vector space 
of dimension $2e$, since $E_R(x) \in \mathscr{A}$ is separable and has degree $p^{2e}$.
\begin{definition}
\label{Heisenberg}
We define the \emph{Heisenberg group} $H_R$ associated with $R(x)$ by setting 
\[
H_R \coloneqq \{(a,b) \in V_R \times 
\mathbb{F} \mid b^p - b = aR(a)\}
\]
 whose group law is given 
 by 
\begin{equation}\label{gpl}
(a,b) \cdot (a',b') = (a+a',\, b+b'+f_R(a,a')). 
\end{equation}
The relation \eqref{a} ensures that this multiplication is well-defined and defines a group structure on $H_R$. 
\end{definition}

The following can be proved by a direct computation.

\begin{lemma}\label{2.3}
For $g=(a,b), g'=(a',b') \in H_R$, we have
\[
  g^{-1} = (-a,\,-b+f_R(a,a)) \quad \text{and} \quad
  [g,g'] = (0,\, \omega_R(a,a')),
\]
where $[g,g'] \coloneqq g g' g^{-1} g'^{-1}$ is the commutator of $g,g'$.
\end{lemma}

\begin{lemma}[{\cite[Lemmas~2.3 and~2.4]{TT}}]
\label{basic}
\begin{enumerate}
    \item[{\rm (1)}] The center $Z(H_R)$ equals $\{ 0 \} \times \F_p$. 
    \item[{\rm (2)}] We have an isomorphism
    $H_R/Z(H_R) \overset{\sim}{\to} V_R$ via the projection $(a,b)\mapsto a$. 
    \item[{\rm (3)}] The map
    \[ H_R\times H_R\to Z(H_R), \quad  (g,g')\mapsto [g,g'] \]
    induces a non-degenerate symplectic pairing
    \[ 
    (H_R/Z(H_R))\times
    (H_R/Z(H_R)) \to Z(H_R). 
    \]
    Moreover, under the identifications $Z(H_R) \simeq \mathbb{F}_p$ and $H_R/Z(H_R) \simeq V_R$ given in {\rm (1) and (2)}, the above pairing is identified with the pairing 
    \[ \omega_R\colon V_R\times V_R\to\mathbb{F}_p, \quad  
    (x,y) \mapsto \omega_R(x,y), \]
    where $\omega_R(x,y)$ is given in \eqref{om}. 
    \item[{\rm (4)}] Assume $p_0=2$. 
    The order of any element of $H_R$ divides $4$. 
\end{enumerate}
\end{lemma}

The Heisenberg group $H_R$ acts on $C_{R,\F}$ by
\begin{equation}\label{std}
(x,y) \cdot (a,b) = (x+a,\, y+b+f_R(x,a)). 
\end{equation}
for $(a,b)\in H_R$ and $(x,y)\in C_{R,\mathbb{F}}$.
This action is well-defined by \eqref{a} and gives a right action of $H_R$ on $C_{R,\mathbb{F}}$. 

Let
\[
\pi \colon H_R \to V_R, \quad  (a,b) \mapsto a. 
\]
By Lemma~\ref{basic} (3),
we can check that the maximal abelian subgroups $\{A\}$ of $H_R$ are in one-to-one correspondence with those subgroups $\{W\}$ of $V_R$ which are maximal totally isotropic with respect to
the symplectic pairing $\omega_R$ via
$A \mapsto \pi(A)$ and $\pi^{-1}(W) \mapsfrom W$.

Let $A \subset H_R$ be a maximal 
abelian subgroup.
From now on, we always assume the following. 

\begin{assumption}
\label{assumption}
$A \subset \mathbb{F}_q^2$.
\end{assumption}

Note that this condition is equivalent to requiring that the automorphism of $C_{R,\mathbb{F}}$ induced by an element of
$A$ is defined over $\mathbb{F}_q$. Indeed, for $(a,b)\in A$, the action \eqref{std} is defined over $\F_q$ if and only if $x+a$ and $y+b+f_R(x,a)$ belong to $ \Gamma(C_R,\mathcal{O}_{C_R})$, which is equivalent to  $a,b\in \F_q$ since $\F_q$ is algebraically closed in $\Gamma(C_R,\mathcal{O}_{C_R})$. 

\begin{remark}
When $p$ is odd, the condition $A\subset \F_q^2$ is equivalent to $\pi(A)\subset \F_q$ \cite[Lemma~2.6]{TT}. When $p$ is even, the former condition is strictly stronger than the latter, as the following example shows. 
\end{remark}
\begin{example}\label{example: GVcurves}
Suppose that $p$ is even, and let 
$R(x)=x^p+sx$ with $s\in \F_q$. Then $E_R(x)=x^{p^2}+x$, 
$f_R(x,y)=xy^p$, and $\omega_R(x,y)=x^py+xy^p$. 
Hence $V_R=\F_{p^2}$, and 
$\overline{A}=\F_p$ is a maximal totally isotropic subspace with respect to $\omega_R$. The inverse image $A=\pi^{-1}(\overline{A})$ is equal to 
\[
A=\{(a,b)\in\F_p\times\F\mid b^p+b=(1+s)a^2\}. 
\]
  Lemma~\ref{Hilbert90} below implies  that $A\subset \F_q^2$ if and only if $\Tr_{q/p}(1+s)=0$. 
\end{example}
\begin{lemma}\label{Hilbert90}
For $x\in \F_q$, there exists a solution $y\in \F_q$ to  
\[y^p-y=x\]
    if and only if $\Tr_{q/p}(x)=0$. 
\end{lemma}
\begin{proof}
    We have a commutative diagram 
    \[
    \xymatrix{
    0\ar[r]&\F_p\ar[d]^-0\ar[r]&\F\ar[d]\ar[r]^{y^p-y}&\F\ar[r]\ar[d]&0\\
0\ar[r]&\F_p\ar[r]&\F\ar[r]^-{y^p-y}&\F\ar[r]&0
    }
    \]
    with exact rows, where the vertical arrows are given by $y\mapsto y^q-y$. By the snake lemma, we obtain an exact sequence 
    \[
    \F_q\xrightarrow{y^p-y}\F_q\xrightarrow{\Tr_{q/p}}\F_p. 
    \]
    The assertion follows. 
\end{proof}

As Example~\ref{example: GVcurves} shows, when $p$ is even, the condition $A\subset \F_q^2$ cannot be detected from its image $\pi(A)$. Nevertheless, the following result holds, although we will not use it in the sequel.
\begin{lemma}
    Suppose $V_R\subset \F_q$. Then there exists a maximal abelian subgroup $A\subset H_R$ satisfying $A\subset \F_q^2$. 
\end{lemma}
\begin{proof}
Let $a \in V_R$, and choose $b \in \F$ such that $(a,b) \in H_R$. We set 
\[
\psi(a)\coloneqq(a^q,b^q)\cdot(a,b)^{-1}=(0,\Tr_{q/p}(aR(a)))\in H_R. 
\]
The above computation shows that $\psi(a)$ does not depend on the choice of $b$. 
By \eqref{a}, it follows that $\psi$ defines an additive map 
\[V_R\to Z(H_R)\cong \F_p.\]
Note that $\psi^{1/2}\colon V_R\to\F_p,\ a \mapsto \Tr_{q/p}(a R(a))^{1/2}$ is $\F_p$-linear. 

From the definition of $\psi$, it follows that a subgroup $A\subset H_R$ satisfies $A\subset\F_q^2$ if and only if $\pi(A)\subset \Ker\psi$, which is in turn equivalent to 
$\pi(A)\subset \Ker\psi^{1/2}$. 
Therefore, to conclude the proof, it suffices to find a maximal totally isotropic subspace of $V_R$ contained in $\Ker\psi^{1/2}$. 

Let $K\coloneqq \Ker\psi^{1/2}$. 
There are two cases: either $K=V_R$ or $\dim_{\F_p}K=\dim_{\F_p}V_R-1$. 
In the case $K=V_R$, any maximal totally isotropic subspace satisfies the required condition. 

Assume now that $K\subsetneq V_R$. In this case, $K$ has odd dimension, and the symplectic form $\omega_R|_K$ is degenerate. Let $\mathrm{Rad}(K)=K \cap K^{\perp}$ denote its radical, 
where 
\[K^{\perp} \coloneqq \{x \in V_R \mid \omega_R(x,y)=0\  \textrm{for any $y \in K$}\}.\]
Then $\dim_{\F_p}\mathrm{Rad}(K)=1$, since $\dim_{\F_p} K^{\perp}=\dim_{\F_p} V_R-\dim_{\F_p} K=1$ and
$\mathrm{Rad}(K)\neq \{0\}$.
Moreover, the form induced by $\omega_R|_K$ on $K/\mathrm{Rad}(K)$ is non-degenerate.
Hence any maximal totally isotropic subspace of $K$
has dimension $\frac12\dim_{\F_p} V_R$.
Therefore, it is also maximal totally isotropic in $V_R$. The assertion follows. 
\end{proof}

\begin{lemma}[{\cite[Lemma~2.5]{TT}}]
\label{F_A}
Assume $A \subset \mathbb{F}_q^2$.
Then there exists an $\F_p$-linearized polynomial $F_A(x) \in \mathscr{A}$ 
such that 
$F_A(x)$ divides $E_R(x)$ and $A = \pi^{-1}(\Ker F_A)$.
\end{lemma}

\subsection{Decomposition of the $\ell$-adic cohomology}

Let $A \subset H_R$ be a maximal abelian subgroup
satisfying $A \subset \mathbb{F}_q^2$.
Take an $\F_p$-linearized polynomial $F_A(x) \in \mathscr{A}$
as in Lemma~\ref{F_A}.

Let
$\phi \colon C_R\to\A^1_{\mathbb{F}_q},\ (x,y)\mapsto x$
denote the projection.
Using \cite[Lemma~2.7]{TT}, we can check that the quotient morphism $C_R\to C_R/A$ is isomorphic over $\F_q$ to the following
morphism:
\begin{equation}\label{phic}
F_A \circ \phi \colon C_R \to 
\A_{\mathbb{F}_q}^1, \quad 
(x,y) \mapsto F_A(x),
\end{equation}
which is a finite \'etale Galois covering with Galois group $A$. 

We decompose the $\ell$-adic cohomology group with compact support
$H^1_c(C_{R,\mathbb{F}},\overline{\Q}_\ell)$ using $F_A\circ\phi$.
To this end, we use the following notation.
For a nontrivial character $\psi \in \mathbb{F}_p^{\vee} \setminus \{1\}$, we define 
\begin{equation}
\label{apsi}
A^{\vee}_{\psi} \coloneqq \{\xi \in A^{\vee} \mid 
\xi|_{Z(H_R)}=\psi\}.
\end{equation}
Here, we regard $\psi$ as a character of
the center $Z(H_R)$ via
$Z(H_R) = \{ 0 \} \times \F_p \simeq \F_p,\ (0,a) \mapsto a$;
see Lemma~\ref{basic} (1).

We define smooth $\overline{\Q}_\ell$-sheaves
on $\A^1_{\mathbb{F}_q}$ using the construction given in
\cite[D\'efinition 1.7]{SommesTrig} and \cite[Proposition~10.1.23]{Fu}.

\begin{definition}
\label{smoothsheaf}
\begin{enumerate}
\item[{\rm (1)}] For $\xi\in A^\vee$, let $\mathcal{Q}_{\xi}$ be the smooth $\overline{\Q}_\ell$-sheaf on $\A^1_{\mathbb{F}_q}$
associated with the $A$-covering
$F_A\circ\phi \colon C_R \to \A_{\mathbb{F}_q}^1$
and  $\xi$.

\item[{\rm (2)}] For $\psi\in \F_p^\vee$, let $\mathcal{L}_{\psi}(x)$ be 
the smooth $\overline{\Q}_{\ell}$-sheaf on $\A^1_{\mathbb{F}_q}$
associated with the Artin--Schreier covering $\A^1_{\mathbb{F}_q} \to \A^1_{\mathbb{F}_q},\ y \mapsto y^p - y$ and  $\psi$.

\item[{\rm (3)}] 
For a polynomial $f(x) \in \F_q[x]$, we define  $\mathcal{L}_{\psi}(f(x)) \coloneqq f^{\ast} \mathcal{L}_{\psi}(x)$, where $f$ is viewed as a  morphism of $\F_q$-schemes $f \colon \A^1_{\mathbb{F}_q} \to \A^1_{\mathbb{F}_q}$. 
\end{enumerate}
\end{definition}

\begin{proposition}\label{14}
\begin{itemize}
\item[{\rm (1)}] 
We have a canonical isomorphism
\[
H_{\rm c}^1(C_{R,\mathbb{F}},\overline{\Q}_{\ell}) \overset{\sim}{\to} 
H^1(\overline{C}_{R,\mathbb{F}},\overline{\Q}_{\ell}).
\]

\item[{\rm (2)}] 
We have an isomorphism
\[ H_{\rm c}^1(C_{R,\mathbb{F}},\overline{\Q}_{\ell}) 
\simeq 
\bigoplus_{\psi \in \mathbb{F}_p^{\vee} \setminus \{1\}}H_{\rm c}^1(\A^1_{\F}, \mathcal{L}_{\psi}(xR(x))).
\]
For a nontrivial character $\psi \in \mathbb{F}_p^{\vee} \setminus \{1\}$, we have an isomorphism
\[
  H_{\rm c}^1(\A^1_{\F}, \mathcal{L}_{\psi}(xR(x)))
  \simeq  
  \bigoplus_{\xi \in A^{\vee}_{\psi}} H_{\rm c}^1(\A^1_{\F}, \mathcal{Q}_{\xi}).
\]

\item[{\rm (3)}] Assume $p_0=2$. Then, for nontrivial characters $\psi,\psi'\in \F_p^{\vee} \setminus \{1\}$,
we have an isomorphism 
\[
H_{\rm c}^1(\A^1_{\F}, \mathcal{L}_{\psi}(xR(x))) \simeq H_{\rm c}^1(\A^1_{\F},\mathcal{L}_{\psi'}(xR(x))).
\]
\end{itemize}
\end{proposition}

\begin{proof}
The assertions (1) and (2) are 
shown in \cite[Corollary~3.6 (1)]{TT} and in \cite[Lemma~2.8]{TT}, respectively.  

We show (3). 
Take $\alpha \in \F_p^{\times}$ such that  
$\psi(x)=\psi'(\alpha x)$ for every $x \in \F_p$. 
Since $\F_p$ is a perfect field of characteristic $2$,
we have $\alpha^{1/2} \in \F_p$.
Then we have isomorphisms
\[ \mathcal{L}_{\psi}(xR(x)) \simeq 
\mathcal{L}_{\psi'}(\alpha xR(x)) =
\mathcal{L}_{\psi'}(\alpha^{1/2} xR(\alpha^{1/2} x)). \]
Making the change of variables $x'=\alpha^{1/2} x$ 
yields the desired isomorphism.
\end{proof}

\begin{lemma}[{\cite[Lemmas~3.3 and~3.4]{TT}}]
\label{rkQ}
For a nontrivial character $\psi\in\mathbb{F}_p^\vee$ and a character $\xi\in A_\psi^\vee$,
we have 
\begin{align*}
  &\dim H^i_{\rm c}(\A^1_{\F}, \mathcal{Q}_\xi) =
  \begin{cases}
  0 & (i \neq 1), \\
  1 & (i = 1), 
  \end{cases} \\
  & \dim H_{\rm c}^1(\A^1_{\F},\mathcal{L}_{\psi}(xR(x))) = p^e. 
\end{align*}
\end{lemma}
Let $n$ be a positive integer such that $\F_q\subset\F_{p_0^n}$. Let 
$\mathrm{Fr}_{p_0^n} \in \Gal(\F/\F_{p_0^n})$
denote the geometric Frobenius element that sends $x \in \F$ to $x^{p_0^{-n}}$. 
We define
\begin{equation}\label{tau}
\tau_{R,\xi,p_0^n} \coloneqq \Tr(\mathrm{Fr}_{p_0^n} \mid 
H_{\rm c}^1(\A^1_{\F},\mathcal{Q}_{\xi})). 
\end{equation}
Since 
$H_{\rm c}^1(\A^1_{\F},\mathcal{Q}_{\xi})$
is one-dimensional by Lemma~\ref{rkQ},
the element $\tau_{R,\xi,p_0^n} $
is the eigenvalue of
$\mathrm{Fr}_{p_0^n}$
acting on 
$H_{\rm c}^1(\A^1_{\F},\mathcal{Q}_{\xi})$.
\begin{lemma}\label{Stone}
Let $\psi \in \F_p^{\vee} \setminus \{1\}$. 
\begin{itemize}
\item[{\rm (1)}] 
There exists a unique irreducible representation $\rho_{\psi}$ of $H_R$ with central character 
$\psi$. The degree of the representation equals $p^e$. 
\item[{\rm (2)}] 
The group $H_R$ naturally acts on $H_{\rm c}^1(\A^1_{\F},
\mathcal{L}_{\psi}(xR(x)))$. It is isomorphic to 
$\rho_{\psi}$ as $H_R$-representations.  
\item[{\rm (3)}] Assume that $p_0=2$ and $H_R \subset \F_q^2$. Then 
the geometric Frobenius $\mathrm{Fr}_q$ acts on $H^1(\overline{C}_{R,\F},\overline{\Q}_{\ell})$ as scalar multiplication by $\tau_{R,\xi,q}$
 for any $\xi \in A_{\psi}^{\vee}$. 
\end{itemize}
\end{lemma}
\begin{proof}
(1) The claim follows from 
the Stone--von Neumann theorem in \cite[Exercise 4.1.8]{Bu}. 

(2) Since $H_R$ acts on $C_{R,\F}$, it acts on the cohomology $H_c^1(C_{R,\F},\overline{\Q}_{\ell})$. 
Recall $\phi \colon C_R \to \A^1_{\F_q},\ 
(x,y) \mapsto x$. 
Then we have 
\[
H_c^1(C_{R,\F},\overline{\Q}_{\ell}) \simeq H_c^1(\A^1_{\F},\phi_\ast\overline{\Q}_{\ell}). 
\]
Since taking $\psi$-isotypic part commutes with taking cohomology, we obtain an isomorphism 
\[H_{c,\psi}^1 \simeq H_{\rm c}^1(\A^1_{\F},\mathcal{L}_{\psi}(xR(x))),
\]
where $H_{c,\psi}^1$ denotes the $\psi$-isotypic part of $H_c^1(C_{R,\F},\overline{\Q}_{\ell})$. 
Thus, the first assertion in (2) follows. 
The center $Z(H_R)$ acts on $H_{\rm c}^1(\A^1_{\F},\mathcal{L}_{\psi}(xR(x)))$ via the character $\psi$. Therefore, by (1) and the semisimplicity, we know that $H_{\rm c}^1(\A^1_{\F},\mathcal{L}_{\psi}(xR(x)))$ is isomorphic to the direct sum of finitely many copies of $\rho_\psi$. 
Since the dimension of $H_{\rm c}^1(\A^1_{\F},\mathcal{L}_{\psi}(xR(x))) $ equals $ p^e$ by Lemma~\ref{rkQ}, the multiplicity should be $1$. 

(3) The $H_R$-representation $H_{\rm c}^1(\A^1_{\F},
\mathcal{L}_{\psi}(xR(x)))$ is irreducible by 
(2). Moreover, if $H_R \subset \F_q^2$, then the actions of $H_R$ and of $\mathrm{Fr}_q$ commute. Therefore, 
by Schur's lemma, 
$\mathrm{Fr}_q$ acts on $H_{\rm c}^1(\A^1_{\F},
\mathcal{L}_{\psi}(xR(x)))$ as scalar multiplication by $\tau_{R,\xi,q}$
 for any $\xi \in A_{\psi}^{\vee}$.  
The claim follows from Proposition~\ref{14}. 
\end{proof}
We define the \emph{$L$-polynomial} of $\overline{C}_R$ by
\[
L_{\overline{C}_R/\F_q}(T) \coloneqq \det(1-\mathrm{Fr}_qT \mid  
H^1(\overline{C}_{R,\F},
\overline{\Q}_{\ell})).
\]

In the following theorem,
we decompose the $L$-polynomial of
the curve $\overline{C}_R$
in terms of the eigenvalues $\tau_{R,\xi,q}$. 

\begin{theorem}[{\cite[Theorem~1.1]{TT}}]
Assume $A \subset \F_q^2$. 
Then we have 
\[
L_{\overline{C}_R/\F_q}
(T)=\prod_{\psi \in \F_p^{\vee}
\setminus \{1\}} \prod_{\xi \in A_{\psi}^{\vee}}
(1-\tau_{R,\xi,q}T). 
\]
\end{theorem}

\section{Geometry of the curve $y^p + y = x^{p+1} + a_0 x^2$}
\label{Section:Geometry}
In this section, 
we study in detail the geometry of the curve $C_S$ defined by the equation
$y^p + y = x S(x)$,
where $S(x)$ is a polynomial of the form $ x^{p} + a_0 x$ ($a_0 \in \F_q$). 

A key fact is that the maximal abelian subgroups of
$H_S$ are isomorphic 
to the group of Witt vectors $W_2(\F_p)$ of length $2$;
see Lemma~\ref{maximalabeliansubgroupwitt}.
We study the $\ell$-adic cohomology of the curve $C_S$
with the aid of the Lang torsor over $W_2$.
Our methods rely on a geometric construction similar to that of Deligne--Lusztig.

Coulter previously studied the number of $\F_p$-rational points on the curve $\overline{C}_S$ by different methods; see \cite{C}.

\subsection{Witt vectors of length $2$}

We briefly recall the ring of Witt vectors of length $2$.
For basic results on Witt vectors,
 see \cite[Chapter II, Section 6]{Se}.

Let $W_2$ be the commutative ring $\F_2$-scheme of
Witt vectors of length $2$.
As an $\F_2$-scheme, it is 
the affine plane 
$\A^2_{\F_2} = \A^1_{\F_2} \times \A^1_{\F_2}$. The ring structure is given by 
\begin{align*}
 (a,b) + (c,d) &= (a+c,\, b+d+ac), \\
 (a,b) \cdot (c,d) &= (ac,\, bc^2+a^2 d).
\end{align*}
Furthermore, we have $ -(a,b) = (a, b+a^2)$. Note that the Frobenius morphism $(a,b) \mapsto (a^2,b^2)$ is a ring homomorphism.

For an $\F_2$-algebra $R$, the additive and multiplicative structures described above endow $W_2(R)$ with the structure of a ring whose multiplicative identity is $(1, 0)$.
An element $(a, b) \in W_2(R)$ is a unit if and only if $a$ is a unit in $R$. In fact,
if $a \in R^{\times}$, then the inverse of $(a, b)$ is given by $(a^{-1}, a^{-4} b)$.
\begin{example}\label{wex}
For $R = \F_2$, the ring $W_2(\F_2)$ is isomorphic to $ \Z/4\Z$.
More generally, for a power $p$ of $2$, 
we have an isomorphism $W_2(\F_p) \simeq \mathcal{O}_K/4\mathcal{O}_K$,
where $\mathcal{O}_K$ is the ring of integers of an unramified extension $K/\Q_2$ of degree $[\F_p : \F_2]$. 
\end{example}
\begin{lemma}\label{free}
The $\Z/4\Z$-module $W_2(\F_p)$ is free. 
\end{lemma}
\begin{proof}
By Example~\ref{wex}, we have an isomorphism 
\[
W_2(\F_p)\simeq  \mathcal{O}_K/4\mathcal{O}_K. 
\]
The assertion follows since $\mathcal{O}_K$ is a finite free $\Z_2$-algebra. 
\end{proof}

Recall that for positive integers $m,n \geq 1$ with $n \mid m$,
the trace map from $\F_{2^m}$ to $\F_{2^n}$ is denoted by 
$\Tr_{2^m/2^n} \colon \F_{2^m} \to \F_{2^n}$.
By an abuse of notation, we use the same symbol $\Tr_{2^m/2^n}$ to denote 
the trace map from $W_2(\F_{2^m})$ to $W_2(\F_{2^n})$. 
It is defined by
\[
\Tr_{2^m/2^n} \colon W_2(\F_{2^m}) \to W_2(\F_{2^n}),\quad (a,b) \mapsto \sum_{i=0}^{\frac{m}{n} - 1}(a^{2^{in}},b^{2^{in}}). 
\]

\begin{lemma}\label{surjTr}
The trace map
$\Tr_{2^m/2^n} \colon W_2(\F_{2^m}) \to W_2(\F_{2^n})$
is surjective.
\end{lemma}

\begin{proof}
We have a commutative diagram of exact sequences
\begin{equation}\label{ses}
\xymatrix{
0 \ar[r] & \F_{2^m} \ar[d]^{\Tr_{2^m/2^n}}\ar[r]^-{i_m} & W_2(\F_{2^m}) \ar[d]^{\Tr_{2^m/2^n}}\ar[r]^-{p_m} & \F_{2^m} \ar[d]^{\Tr_{2^m/2^n}}\ar[r] & 0 \\
0 \ar[r] & \F_{2^n} \ar[r]^-{i_n} & W_2(\F_{2^n}) \ar[r]^-{p_n} & \F_{2^n} \ar[r] & 0, 
}
\end{equation}
where $i_m$ and $i_n$ are given by $x \mapsto (0,x)$, 
and $p_m$ and $p_n$ are given by $(x,y) \mapsto x$. 
By the snake lemma, the assertion follows from the surjectivity of
$\Tr_{2^m/2^n} \colon \F_{2^m} \to \F_{2^n}$.
\end{proof}

In this paper, we use the following notation.
We fix a faithful character
\[ \xi_2 \in W_2(\F_2)^{\vee} \simeq (\Z/4\Z)^{\vee}. \]
We set
$\sqrt{-1} \coloneqq \xi_2(1,0) \in \overline{\Q}_{\ell}^{\times}$,
which is a primitive $4$-th root of unity in $\overline{\Q}_{\ell}$.
For a positive integer $n \ge 1$,
we set
\[ \xi_{2^n} \coloneqq \xi_2 \circ \Tr_{2^n/2} \in W_2(\F_{2^n})^{\vee}. \]
It follows from Lemma~\ref{surjTr} that $\xi_{2^n}$ is a character of order $4$.

For $(a,b) \in W_2(\F_p)$, 
we define a group homomorphism 
\begin{equation}\label{trab}
\Tr_{(a,b)} \colon W_2(\F_p) \to W_2(\F_2), \quad 
(x,y) \mapsto\Tr_{p/2}((a,b)\cdot (x,y)). 
\end{equation}

\begin{lemma}\label{trw}
We have an isomorphism  
\[
\Phi \colon W_2(\F_p) \overset{\sim}{\to} W_2(\F_p)^{\vee}, \quad 
(a,b) \mapsto \xi_2 \circ \Tr_{(a,b)}. 
\]
For a character $\xi \in W_2(\F_p)^{\vee}$, let  $(a_{\xi},b_{\xi}) \in W_2(\F_p)$ be the unique element 
such that $\xi = \xi_2 \circ \Tr_{(a_{\xi},b_{\xi})}$. 
Then $\xi$ has order $4$ if and only if $a_{\xi} \neq 0$. 
\end{lemma}

\begin{proof}
To prove that $\Phi$ is an isomorphism, it suffices to show that $\Phi$ is injective.
Assume $\Phi(a,b)=1$. Note that 
$(a,b)\cdot (x,y)=(ax, bx^2+a^2y)$ for 
$(x,y) \in W_2(\F_p)$. 
Then we have 
\[
1=\xi_2(\Tr_{p/2}((a,b) \cdot (0,y)))
=\xi_2(0,\Tr_{p/2}(a^2y))
\]
 for any $y \in \F_p$. 
Since $\xi_2|_{\{0\} \times \mathbb{F}_2}$ is 
nontrivial, we obtain $a=0$. 
Hence we have 
\[
1=\xi_2(\Tr_{p/2}((0,b) \cdot (x,0)))
=\xi_2(0,\Tr_{p/2}(bx^2))
\]
for any $x \in \F_p$, thus $b=0$. 

Since $\Phi$ is an isomorphism of abelian groups, the character $\xi$ has order $4$ if and only if the element $(a_\xi,b_\xi) \in W_2(\F_p)$ has order $4$. 
This condition is equivalent to $a_{\xi} \neq 0$. 
Hence the final assertion follows. 
\end{proof}

\subsection{The Heisenberg group $H_S$}

Here we explicitly calculate the Heisenberg group $H_S$ for the binomial
\[ 
S(x) = x^p + a_0 x \in \mathscr{A}. 
\]

By \eqref{-aa}, we have 
\[
  E_S(x) = (x^p + a_0 x)^p + (a_0 x)^p + x
  = x^{p^2} + 2 (a_0 x)^p + x
  = x^{p^2} + x, 
\]
and hence  
\[
  V_S = \Ker E_S = \{x \in \mathbb{F} \mid x^{p^2} + x = 0 \} = \F_{p^2}. 
\]
By \eqref{-a}, we have $f_S(x,y) = x y^p$ and
$\omega_S(x,y) = f_S(x,y) + f_S(y,x) = x^p y + x y^p$. 
Thus, by Definition \ref{Heisenberg}, the Heisenberg group $H_S$
is given by
\[
  H_S = \{ (a,b) \in \F_{p^2} \times \F \mid b^p + b = a S(a)=a^{p+1} + a_0 a^2 \}
\]
and the group law is given by
\[ (a,b) \cdot (a',b')=(a+a',\,b+b'+aa'^p). \] 

Let $\xi\in V_S \setminus \{0\}$. 
For $x,y \in \F_p$,
we have 
$\omega_S(x \xi,y \xi) = (x^p y + x y^p )\xi^{p+1}=
(xy + xy)\xi^{p+1} = 0$.
Hence the subspace 
$\mathbb{F}_p\, \xi \subset V_S=\F_{p^2}$ is a maximal totally 
isotropic subspace with respect to $\omega_S$. 
Let $\pi_S \colon H_S \to V_S,\ (a,b) \mapsto a$. 
We take $\xi=1$ and put
\begin{equation}
\label{maximalabeliansubgroup}
A_S \coloneqq \pi_S^{-1}(\mathbb{F}_p)
  = \{(a,b) \in \F_{p} \times \F \mid b^p + b = (1 + a_0) a^2 \},
\end{equation}
which is a maximal abelian subgroup of $H_S$ (cf.\ Example \ref{example: GVcurves}).
 
\begin{lemma}
\label{maximalabeliansubgroupwitt}
 For $(1,\eta) \in A_S$, 
we have an isomorphism of abelian groups 
 \begin{equation}\label{aw}
f_{\eta} \colon A_S \overset{\sim}{\to} W_2(\mathbb{F}_p), \quad 
 (a,b) \mapsto (a,\, b+a^2\eta).
 \end{equation}
 \end{lemma}

 \begin{proof}
 Let $(a,b) \in A_S$. 
From $b^p+b=a S(a)=a^2 S(1)=(a^2 \eta)^p+a^2 \eta$, it follows that the map $f_{\eta}$ is well-defined. 
 Since $f_{\eta}$ is injective and 
 $|A_S|=|W_2(\F_p)|=p^2$, this map is bijective. The map $f_{\eta}$ is a group homomorphism, since 
\begin{align*}
f_{\eta}((a, b) \cdot (a',b'))&=
f_{\eta}(a+a',b+b'+aa'^p)\\
&=(a+a', b+b'+aa'+(a+a')^2 \eta)\\
&=
(a,b+a^2 \eta)+(a',b'+a'^2 \eta)=f_{\eta}(a,b)
+f_{\eta}(a',b'). 
\end{align*}
\end{proof}

\subsection{The Lang torsor over $W_2$}
\label{LangTorsor}

Since $W_2$ is a smooth group scheme over $\F_2$,
the Lang torsor can be constructed via the Frobenius morphism.

\begin{definition}
The morphism over $\F_p$, 
\[ L_p \colon W_{2,\F_p} \to W_{2,\F_p}, \quad 
(x,y) \mapsto (x^p,y^p)- (x,y),  
\]
is called the \emph{Lang torsor} over $W_{2,\F_p}$. 
This is a finite \'etale Galois covering whose Galois group is $W_2(\mathbb{F}_p)$.
\end{definition}

Explicitly, we have  
\begin{equation}
\label{coo-1}
  (x^p,y^p) - (x,y) = (x^p,y^p) + (x,y+x^2)
  = (x^p+x,\, y^p+y+x^{p+1}+x^2). 
\end{equation}

For $\alpha \in \mathbb{F}_q$,
we have a closed embedding
\[ \A^1_{\mathbb{F}_q} \hookrightarrow W_{2,\F_q}, \quad s \mapsto (1,\a) \cdot (s,0)=(s,\a s^2). \]

Let $T_{\a} \subset W_{2,\F_q}$ denote the closed subscheme defined by this embedding. 
The inverse image
$L_p^{-1}(T_{\a}) \subset W_{2,\F_q}$ is given by
\begin{equation}\label{coo}
  L_p^{-1}(T_{\a}) =
  \big\{ (x,y) \in W_{2,\F_q} \, \mid  
       y^p + y + x^{p+1} + x^2 = \alpha (x^p + x)^2 \big\}.
\end{equation}
The group $W_2(\mathbb{F}_p)$ acts on
the closed subscheme $L_p^{-1}(T_{\a}) \subset W_{2,\F_q}$ naturally. 

The following lemmas show that 
$L_p^{-1}(T_{\a})$ is isomorphic to the curve $C_S$
for an appropriate choice of $\alpha \in \F_q$.

\begin{lemma}
\label{lemmaalpha}
Assume $A_S \subset \F_q^2$.
Then there exists $\alpha \in \F_q$ satisfying
$\a^{p^{-1}}+\a = a_0+1$.
\end{lemma}

\begin{proof}
By $A_S \subset \F_q^2$ and \eqref{maximalabeliansubgroup}, 
there exists 
$\eta\in\F_q$ such that $\eta^p + \eta =a_0+1$.  
Then $\alpha \coloneqq \eta^p$ satisfies the desired condition.
\end{proof}

\begin{lemma}\label{dc}
Let $\alpha\in \F_q$ satisfy $\a^{p^{-1}}+\a=a_0+1$. 
Then the map 
\[
C_S\to W_{2,\F_q}, \quad  (x,z) \mapsto 
\bigl(x, z+\a^{p^{-1}} x^2\bigr)
\]
induces an isomorphism 
\[
\iota \colon C_S \xrightarrow{\sim} L_p^{-1}(T_{\a}). 
\]
\end{lemma}

\begin{proof}
The closed subscheme $L_p^{-1}(T_\a)$ is defined by the equation 
\[
y^p+y+x^{p+1}+x^2=\a x^{2p}+\a x^2. 
\]
Let $y=z+\a^{p^{-1}} x^2$. Since $\a^{p^{-1}}+\a= a_0+1$, we have 
\[
z^p+z=y^p+y+\a x^{2p}+\a^{p^{-1}}
x^2=x^{p+1}+a_0 x^2. 
\]
This shows that the map $C_S \to W_{2,\F_q}$ factors through $C_S \to L_p^{-1}(T_{\a})$. The inverse map is given by $L_p^{-1}(T_{\a}) \to C_S,\
(x,y) \mapsto (x,y+\a^{p^{-1}} x^2)$. 
\end{proof}

\subsection{Identification of $\mathcal{Q}_\xi$}

Recall that $S(x) = x^p + a_0 x \in \mathscr{A}$
and $A_S=\pi_S^{-1}(\F_p)$. We assume $A_S \subset \F_q^2$. 
Since $\pi_S(A_S)=\F_p$, we may take $F_A(x)$ in Lemma~\ref{F_A} to be $x^p+x$. Then the map~\eqref{phic} coincides with
\begin{equation}\label{pilf}
\nu_S \colon C_S \to \A_{\F_q}^1, \quad (x,y) \mapsto x^p+x,
\end{equation}
which is an $A_S$-covering.
By Lemma~\ref{lemmaalpha}, there exists $\alpha \in \F_q$ satisfying
$\alpha^{p^{-1}} + \alpha=a_0 + 1$.
Note that $(1,\a^{p^{-1}}) \in A_S$. 
By Lemma~\ref{maximalabeliansubgroupwitt},
we identify $A_S$ with $W_2(\mathbb{F}_p)$ via the isomorphism
\begin{equation}\label{aaw}
f_{\alpha^{p^{-1}}} \colon A_S \xrightarrow{\sim} W_2(\mathbb{F}_p), \quad 
 (a,b) \mapsto \bigl(a,\, b + a^2 \alpha^{p^{-1}}\bigr). 
 \end{equation}

\begin{definition}\label{sheafW}
Let $\xi \in W_2(\F_p)^{\vee}$
be a  character, which is regarded as a character of $A_S$ via the isomorphism $f_{\alpha^{p^{-1}}}$.
\begin{enumerate}
\item[{\rm (1)}] 
Let $\mathcal{Q}_{\xi}$ be the smooth $\overline{\Q}_{\ell}$-sheaf on $\A_{\F_q}^1$
defined by the $A_S$-covering $\nu_S$ and $\xi$
as in Definition \ref{smoothsheaf}.

\item[{\rm (2)}] 
Let $\mathcal{L}_{\xi}(z,w)$ denote the smooth $\overline{\Q}_{\ell}$-sheaf on $W_{2,\F_p}$ associated with
$\xi$ and the Lang torsor $L_p \colon W_{2,\F_p} \to W_{2,\F_p}$.
For two polynomials $f(s),g(s) \in \F_q[s]$,
we consider the morphism
\[ f \times g \colon \A_{\F_q}^1 \to W_{2,\F_q}, \quad 
s \mapsto (f(s),g(s)). \] 
Then we set 
$\mathcal{L}_{\xi}(f(s),g(s)) \coloneqq (f \times g)^{\ast} \mathcal{L}_{\xi}(z,w)$,
which is a smooth $\overline{\Q}_{\ell}$-sheaf on $\A_{\F_q}^1$.

\end{enumerate}
\end{definition}

\begin{lemma}\label{QL}
  For a character $\xi\in W_2(\F_p)^\vee$,  we have  an isomorphism 
\[\mathcal{Q}_{\xi} \simeq 
\mathcal{L}_{\xi}(s,\a s^2).\] 
\end{lemma}
\begin{proof}
    We consider the following commutative diagram 
\[
\xymatrix{
C_S \ar[dr]_-{\nu_S}\ar[r]^-{\simeq}_-{\iota} & L_p^{-1}(T_{\a}) \ar[d]\ar[rr] &  & W_{2,\F_q} \ar[d]^-{L_p}\\
& \A_{\F_q}^1 \ar[rr]^-{s \mapsto (s,\a s^2)} & & 
W_{2,\F_q}, 
}
\]
where $\iota$ is given by 
$(x,z) \mapsto (x,z+\alpha^{p^{-1}} x^2)$ as 
in Lemma~\ref{dc}. 
To prove the assertion, it suffices to show that $\iota$ is equivariant with respect to $f_{\alpha^{p^{-1}}}$ in \eqref{aaw}. 
We write $\circledast$ and 
$\ast$ for the action of $A_S$ on $C_S$ and 
the action of $W_2(\F_p)$ on $L_p^{-1}(T_{\a})$, respectively. 
Then for $(x,z) \in C_S$ and $(a,b) \in A_S$, 
we have  
\begin{align*}
\iota((x,z) \circledast (a,b))&=
\iota(x+a,z+b+ax) \\
&=
(x+a,z+b+ax+\alpha^{p^{-1}}(x+a)^2) \\
&=(x,z+\alpha^{p^{-1}}x^2) \ast 
(a,b+\alpha^{p^{-1}} a^2) \\
&=\iota(x,z) \ast f_{\alpha^{p^{-1}}}(a,b). 
\end{align*}
Thus we obtain an isomorphism $\mathcal{Q}_{\xi} \simeq 
\mathcal{L}_{\xi}(s,\alpha s^2)$. 
\end{proof}

For a positive integer $n\geq1$, we set 
\begin{equation}\label{defGr}
    G_{\xi_{2^n}} \coloneqq -\sum_{x\in\F_{2^n}}\xi_{2^n}(x,0). 
\end{equation}
\begin{lemma}\label{HDGr}
    We have  
    \[
    G_{\xi_{2^n}}=(-1-\sqrt{-1})^n. \]
\end{lemma}
\begin{proof}
By the Grothendieck--Lefschetz trace formula, we have  
\[
G_{\xi_{2^n}}=\sum_i(-1)^i\Tr({\rm Fr}_{2^n} \mid H^i_c(\A^1_{\F},\mathcal{L}_{\xi_2}(s,0))). 
\]
Therefore, it suffices to show that $H^i_c$ vanishes for $i\neq1$ and has dimension $1$ for $i=1$ since $\mathrm{Fr}_{2^n}=\mathrm{Fr}_2^n$ and $G_{\xi_2}=-1-\sqrt{-1}$. 
To this end, we consider the case $S(x)=x^p+x$, 
for which we may take $\alpha=0$.  Then the assertion follows from Lemmas~\ref{rkQ} and~\ref{QL}. 
\end{proof}

\subsection{Calculation of Frobenius eigenvalues}

In this subsection, we compute $\tau_{S,\xi,q}$ by Lemma~\ref{QL}. 
Let $\xi \in W_2(\F_p)^{\vee}$ be a faithful character.
By Lemma~\ref{trw},
we can write
$\xi=\xi_2 \circ \Tr_{(a_{\xi},b_{\xi})}$
for a unique element $(a_{\xi},b_{\xi}) \in W_2(\F_p)^{\times}$.

For $u,v \in \F_{2^n}$, we set
\begin{equation}\label{cplf}
c_{\xi}(u,v) \coloneqq (u + (b_{\xi}/a_{\xi}^2))^{1/2} + a_{\xi} v \in \F_{2^n}. 
\end{equation}
We also define $c_\xi(u)\coloneqq c_\xi(u,0)$.

\begin{lemma}\label{plf}
Let $u,v \in\F_q$. 
\begin{itemize}
\item[{\rm (1)}] 
We have an isomorphism
\[ \mathcal{L}_{\xi}(s,u s^2+v s) \simeq 
\mathcal{L}_{\xi_2}(a_\xi s+c_{\xi}(u,v),0) \otimes \mathcal{L}_{\xi_2^{-1}}(c_{\xi}(u,v),0). \] 

\item[{\rm (2)}]
Let $n$ be a positive multiple of $[\F_q:\F_2]$. 
We have
\begin{align*}
\Tr(\mathrm{Fr}_{2^n} \mid H_{\rm c}^1(\A^1_{\F},\mathcal{L}_{\xi}(s,u s^2+v s)))=
\xi_{2^n}^{-1}(c_{\xi}(u,v),0) \cdot G_{\xi_{2^n}}. 
\end{align*}
\end{itemize}
\end{lemma}

\begin{proof}
(1) For brevity, we denote 
$a_{\xi}$, $b_{\xi}$ and  
$c_{\xi}(u,v)$ by $a$, $b$
and $c$, respectively.
Since $(a,b) \cdot (x,y)=(ax,b x^2+a^2 y)$ in 
$W_2$ and $\xi=\xi_2 \circ \Tr_{(a,b)}$, 
we have 
\begin{gather}\label{sa}
\begin{aligned}
\mathcal{L}_{\xi}(s, u s^2 + v s)
&\simeq \mathcal{L}_{\xi_2 \circ \Tr_{p/2}}(a s, (a^2 u + b) s^2 + a^2 v s) \\
&\simeq \mathcal{L}_{\xi_2}(a s, (a^2 u + b)  s^2 + a^2 v s) \\
& \simeq \mathcal{L}_{\xi_2}(a s, a c s) \\
& \simeq \mathcal{L}_{\xi_2}(
(a s + c, 0) + (c, c^2)) \\
&\simeq \mathcal{L}_{\xi_2}(a s + c, 0)
\otimes \mathcal{L}_{\xi_2^{-1}}(c, 0), 
\end{aligned}
\end{gather}
where the third isomorphism uses 
$\mathcal{L}_{\xi_2}(f^2,g^2)
\simeq \mathcal{L}_{\xi_2}(f,g)$,  
 and the 
last isomorphism uses 
$\mathcal{L}_{\xi_2}((f_1,g_1)+(f_2,g_2))
\simeq \mathcal{L}_{\xi_2}(f_1,g_1)
\otimes  \mathcal{L}_{\xi_2}(f_2,g_2)$ and  
$-(x,0)=(x,x^2)$. 

(2) Since $H_{\rm c}^i(\A_{\F}^1,\mathcal{L}_{\xi}(s,u s^2+v s))=0$ for $i \neq 1$,
the claim follows from (1), $a_{\xi} \neq 0$ and the Grothendieck--Lefschetz trace formula.
\end{proof}

Recall that $S(x) = x^p + a_0 x \in \mathscr{A}$
and $A_S=\pi_S^{-1}(\F_p) \subset \F_q^2$. 
Choose $\alpha \in \F_q$ such that $\alpha^{p^{-1}} + \alpha =a_0+ 1$.
Let $\xi \in W_2(\F_p)^{\vee}$ be a faithful 
character.
Let  $n$ be a positive integer such that 
$\F_q \subset \F_{2^n}$. Let $\tau_{S,\xi,2^n}$ be defined as in \eqref{tau}.

\begin{theorem}\label{hc}
Let the notation and assumptions be as above.  
 Then we have 
 \[
\tau_{S,\xi,2^n} = \xi_{2^n}(c_{\xi}(\a),0)^{-1} \cdot
(-1-\sqrt{-1})^n. 
\] 
\end{theorem}

\begin{proof}
 The claim follows from Lemmas~\ref{QL}, \ref{HDGr}, and~\ref{plf} (2). 
\end{proof}

 \begin{remark}\label{quartic}
 In the following, we explain that the cohomology group 
 \[
 H_{\rm c}^1(\A_{\F}^1,\mathcal{L}_{\xi}(s,u s^2+v s)) \quad \textrm{for $u, v \in 
 \F_{2^n}$}
 \]
 in Lemma~\ref{plf} 
 is regarded as a part of $H^1$ of a supersingular elliptic curve. 
 We simply write $c$ for $c_{\xi}(u,v)$ in \eqref{cplf}.  
By Lemma~\ref{plf} (1), 
 we have isomorphisms
 \begin{align*}
 H_{\rm c}^1(\A_{\F}^1,\mathcal{L}_{\xi}(s,u s^2+v s))
 &\simeq H^1_{\rm c}(\A_{\F}^1,\mathcal{L}_{\xi_2}(s_1+c,0) \otimes \mathcal{L}_{\xi_2^{-1}}(c,0)) \\
 &\simeq  H^1_{\rm c}(\A_{\F}^1,\mathcal{L}_{\xi_2}(s_1,cs_1))
 \simeq  H^1_{\rm c}(\A_{\F}^1,\mathcal{L}_{\xi_2}(s_1,(cs_1)^2)). 
 \end{align*}   
 Here we set $s_1=a_\xi s$. 
 Let $E_c$ denote the supersingular elliptic 
 curve over $\F_{2^n}$ with affine equation
\[ y^2+y=x^3+(c^2+c+1) x^2. \]
 By Proposition~\ref{14} (2) and Lemma~\ref{QL} with $p=2$ and $\a=c^2$, we have an isomorphism
 \begin{align*}
   H^1(E_{c,\F},\overline{\Q}_{\ell}) & \simeq H^1_{\rm c}(\A_{\F}^1,\mathcal{L}_{\xi_2}(s_1,(cs_1)^2)) \oplus H^1_{\rm c}(\A_{\F}^1,\mathcal{L}_{\xi_2^{-1}}(s_1,(cs_1)^2)). 
\end{align*}
From this decomposition together with Lemma~\ref{plf} (2), we obtain 
\[
 L_{E_c/\F_{2^n}}(T)=
 \bigl(1-\xi_{2^n}(c,0)^{-1}\cdot G_{\xi_{2^n}} T\bigr)
 \bigl(1-\xi_{2^n}(c,0)\cdot G_{\xi^{-1}_{2^n}} T\bigr). 
 \]
If $\Tr_{2^n/2}(c)=1$, then $\xi_{2^n}(c,0)$ is a primitive $4$-th root of unity.
Hence $E_c$ is a quartic twist of $E_0$.
 \end{remark}
 
 \begin{remark}\label{relation with KT}
 Let $A_1>A_2>\cdots>A_n\ge1$ $(n\ge2)$ be odd integers. 
In \cite[Section~9.2]{KT}, Katz and Tiep study the local system $\mathcal{W}(A_1,\dots,A_n)$ on $X=\A^n_{\F_q}$, whose coordinates are written as $(s,t_2,\dots,t_n)$,  defined as follows. Consider the $\F_q$-morphisms
\[
X\xleftarrow{{\rm pr}_2}\A^1_{\F_q}\times_{\F_q}X\xrightarrow{f} W_{2,\F_q}, 
\]
where $f$ is defined by 
\[f(x,s,t_2,\dots,t_n)=\biggl(sx,x^{A_1}+\sum_{j=2}^{n}t_jx^{A_j}\biggr).\]
Then the $\overline{\Q}_\ell$-sheaf $\mathcal{W}(A_1,\dots,A_n)$ on $X$ is defined, up to twisting by a rank-one sheaf,  by 
\[
\mathcal{W}(A_1,\dots,A_n)=R^1{\rm pr}_{2!}\mathcal{L}_\xi(f)
\]
for a fixed faithful character $\xi\colon W_2(\F_2)\to \overline{\Q}_{\ell}^\times$. Here we set $\mathcal{L}_{\xi}(f)\coloneqq f^\ast \mathcal{L}_{\xi}(z,w)$ (cf.\ Definition \ref{sheafW} (2)). 

The $\xi$-isotypic part of the cohomology group in Lemma~\ref{plf} (2) can be identified with a fiber of this sheaf in the case where $n=1$ and $A_1=1$, which the authors do not seem to treat (note that $\mathcal{L}_\xi(s,us^2+vs)\cong \mathcal{L}_\xi(s,(v+u^{1/2})s)$). 

Alternatively, suppose that the $A_i$ are of the form $p^{n_i}+1$. In this case, the finite Galois covering $Y\to \A^1_{\F_q}\times_{\F_q}X$ that trivializes $\mathcal{L}_\xi(f)$ is a family of van der Geer--van der Vlugt curves, in the sense that $Y$ is smooth over $X$ and each fiber is isomorphic to the curve $C_R$ for some additive polynomial $R$ (as follows from a direct computation using \eqref{coo-1}). Thus the sheaf $\mathcal{W}(A_1,\dots,A_n)$ can be studied via  a computation of  the cohomology of such curves. 
 \end{remark}
 
\section{Quotients of van der Geer--van der Vlugt curves}
\label{Section:Quotients}
Let $R_1(x)$ and $R_2(x)$ be two $\F_p$-linearized polynomials. 
In Subsection \ref{R1R2}, we consider a specific situation in which there is an induced quotient map $C_{R_1}\to C_{R_2}$ between the van der Geer--van der Vlugt curves. In Subsection \ref{RS}, we specialize to the case where $R_2(x)$ is of the form $x^p+s_0x$, in order to relate general van der Geer--van der Vlugt curves to the curves studied in Section \ref{Section:Geometry}.

\subsection{Maps between van der Geer--van der Vlugt curves}
\label{R1R2}
 In this subsection, we consider the following situation. 
Let $R_1(x), R_2(x), f(x) \in \mathscr{A} \setminus \{0\}$. 
Suppose that $f(x)$ is separable and that there exists a polynomial 
$\Delta(x) \in \F_q[x]$ such that 
$\Delta(0)=0$ and 
\begin{equation}\label{das}
\Delta(x)^p+\Delta(x)=x R_1(x)+f(x) R_2(f(x)).
\end{equation} 
We define the polynomial $\delta(x,y)\in\F_q[x,y]$ by 
\begin{equation}\label{das2}
    \delta(x,y)=\Delta(x+y)+\Delta(x)+\Delta(y)+f_{R_1}(x,y)+
f_{R_2}(f(x),f(y)). 
\end{equation}
We assume $\deg R_i=p^{e_i}$ with $e_i\geq1$. 
\begin{lemma}\label{lmPdel}
Let the notation and assumptions be as above.   
We have $\delta(0,0)=0$ and 
\begin{equation}\label{Pdel}
\delta(x,y)^p+\delta(x,y)=
x^{p^{e_1}} E_{R_1}(y)+
f(x)^{p^{e_2}} E_{R_2}(f(y)). 
\end{equation}  
Moreover, we have 
$\delta(x,y+z)=\delta(x,y)+\delta(x,z)$, $\delta(x+y,z)=\delta(x,z)+\delta(y,z)$, and $\delta(ax,y)=a\delta(x,y)=\delta(x,ay)$ for $a\in \F_p$. 
\end{lemma}
\begin{proof}
  The assertion $\delta(0,0)=0$ is clear 
since $\Delta(0)=0$. 
By using \eqref{a} and \eqref{das}, we compute 
\begin{align*}
\delta(x,y)^p+\delta(x,y)&=
(x+y) R_1(x+y)
+f(x+y)R_2(f(x+y))\\
&\quad +xR_1(x)+f(x) R_2(f(x))+y R_1(y)+f(y) R_2(f(y))
\\
&\quad +x^{p^{e_1}} E_{R_1}(y)+xR_1(y)+yR_1(x) \\
&\quad +f(x)^{p^{e_2}} E_{R_2}(f(y))+f(x) R_2(f(y))+f(y)R_2(f(x))\\
&=x^{p^{e_1}} E_{R_1}(y)+f(x)^{p^{e_2}} E_{R_2}(f(y)).   
\end{align*} 
Thus the equality \eqref{Pdel} follows.  We prove $\delta(x,y+z)=\delta(x,y)+\delta(x,z)$. The remaining equalities follow similarly. Set $g(x,y,z) \coloneqq \delta(x,y+z)-\delta(x,y)-\delta(x,z)$. 
Since \eqref{Pdel} and $f(x), E_{R_i}(x) \in \mathscr{A}\  (i=1,\, 2)$, this satisfies 
\[g(x,y,z)^p-g(x,y,z)=0.\]
Hence, $g(x,y,z)$ is constant. Since $g(0,0,0)=0$ by $\delta(0,0)=0$, we obtain $g(x,y,z)=0$, as desired. 
\end{proof}
Consider the $\F_p$-linear map $f\colon \F\to \F$
defined by $f(x) \in \mathscr{A}$. 
\begin{lemma}\label{Kerf}

 \begin{itemize}
 \item[{\rm (1)}] We have 
\[\delta(x,x)=f_{R_1}(x,x)+f_{R_2}(f(x),f(x)). \]
\item[{\rm (2)}] We have $\Ker f \subset V_{R_1}$.  For $a \in \Ker f$, we have $\delta(x,a)=0$. In particular, we have $f_{R_1}(a,a)=0$. 
\item[{\rm (3)}] 
Consider the subspace $(\Ker f)^{\perp}\subset V_{R_1}$ defined by 
\[
 (\Ker f)^{\perp}\coloneqq \{x \in V_{R_1} \mid 
\omega_{R_1}(x,y)=0\ 
\textrm{for every $y \in \Ker f$}\}.
\]
For $a\in (\Ker f)^{\perp}$, we have $f(a) \in V_{R_2}$ and 
$\delta(x,a) = 0$. In particular, $f_{R_1}(a,a)
=f_{R_2}(f(a), f(a))$. 

\end{itemize}
\end{lemma}
\begin{proof}(1) The assertion follows by setting $x=y$ in \eqref{das2} since $\Delta(0)=0$.

(2) Let $a\in \Ker f$. By \eqref{Pdel}, 
 we have 
\begin{equation}\label{dela}
    \delta(x,a)^p+\delta(x,a)=x^{p^{e_1}}E_{R_1}(a).
\end{equation}
Assume that $E_{R_1}(a)$ is non-zero. Then \eqref{dela} and $e_1>0$ imply that 
$\delta(x,a)$ is a non-constant polynomial. Write $\delta(x,a)=\sum_{i=m}^nb_ix^i$ with $b_m,b_n\neq0$ and $n>0$. Then the highest degree appearing in $\delta(x,a)^p+\delta(x,a)$ is $pn$, while the lowest one is $m$. Note that $pn>m$. On the other hand, 
\eqref{dela} implies that $pn=m=p^{e_1}$.  This is a contradiction.  Thus we must have $E_{R_1}(a)=0$, which shows $a\in V_{R_1}$. 

From the equality $E_{R_1}(a)=0$ and \eqref{dela}, it follows that $\delta(x,a)\in\F_p$.
Since $\delta(0,a)=0$ by \eqref{das2}, we obtain $\delta(x,a)=0$.
The last assertion follows from {\rm (1)}, together with $f(a)=0$ and $\delta(a,a)=0$.

(3) Let $a\in (\Ker f)^{\perp}$. 
Since $a \in V_{R_1}$ and \eqref{Pdel}, we have 
\begin{equation}\label{dxa}
    \delta(x,a)^p+\delta(x,a)=f(x)^{p^{e_2}}E_{R_2}(f(a)).
\end{equation} 
First we prove that $\delta(x,a)$ can be written as $g(f(x))$ for some $g\in \F_q[x]$. Let $b\in \Ker f$. Since \eqref{das2} and $a \in (\Ker f)^{\perp}$, it follows that 
\[ 
\delta(b,a)-\delta(a,b)=\omega_{R_1}(a,b)=0. 
\]
Since $\delta(a,b)=0$ by $b \in \Ker f$ and (2), we obtain $\delta(b,a)=0$. 
We note that $\delta(x,a) \in \mathscr{A}$ by Lemma~\ref{lmPdel}.  
Since $f(x)$ is separable, this implies that $\delta(x,a)=g(f(x))$, as claimed. Since $\deg f>0$, substituting $X=f(x)$ in \eqref{dxa}, we obtain 
\[
g(X)^p+g(X)=X^{p^{e_2}}E_{R_2}(f(a)). 
\]
Then, by a similar argument as above, we conclude that $E_{R_2}(f(a))=0$ and $g(X)=0$. Equivalently, we have $f(a)\in V_{R_2}$ and $\delta(x,a)=0$. 
The last equality follows from (1) and 
$\delta(a,a)=0$. 
\end{proof}

The degrees of $R_1$, $R_2$ and $f$ satisfy the following relation.
\begin{lemma}\label{degcomp}
Let the notation and assumptions be as above. 
We have 
\[
\deg R_1=\deg f \cdot \deg R_2. 
\]
\end{lemma}
\begin{proof}
We write 
$R_2(x)=\sum_{i=0}^{e_2} b_i x^{p^i}$ and 
$f(x)=\sum_{i=0}^r c_i x^{p^i}$ with 
$c_r \neq 0$. Since $f$ is separable, 
we have $c_0 \neq 0$. 
We have 
\begin{align*}
f(x) R_2(f(x))&= \sum_{i=0}^{e_2} 
\sum_{j=0}^r \sum_{k=0}^r b_i c_j^{p^i} c_k 
x^{p^{i+j}+p^k} \\
& \sim R_3(x)\coloneqq \sum_{i=0}^{e_2} 
\sum_{j=0}^r \sum_{k=0}^r (b_i c_j^{p^i} c_k)^{p^{-r_{i,j,k}}} 
x^{p^{-r_{i,j,k}}(p^{i+j}+p^k)}, 
\end{align*}
where $\sim$ denotes the equivalence defined in Section \ref{Introduction} and 
$r_{i,j,k}\coloneqq \min\{i+j,k\}$.
Since $b_{e_2} c_r^{p^{e_2}} c_0 \neq 0$ and 
$p^{-r_{i,j,k}}(p^{i+j}+p^k) < p^{e_2+r}+1$
for any $(i,j,k)\neq (e_2,r,0)$, 
we obtain $\deg R_3=p^{e_2+r}+1$. 
By \eqref{das}, there exists a polynomial $d(x) \in \F_q[x]$ such that 
\[
d(x)^p+d(x)=x R_1(x)+R_3(x). 
\]

Note that 
$\deg (d^p+d)$ is divisible by $p$  if $d(x)^p+d(x)\neq0$. Since $\deg x R_1(x)=p^{e_1}+1$ and 
$\deg R_3=p^{e_2+r}+1$ are prime to $p$, we must have 
$p^{e_1}+1=p^{e_2+r}+1$. This proves the claim. 
\end{proof}

From the above discussion, we obtain the following consequence. 
\begin{proposition}\label{KerfA}
\begin{itemize} 
    \item[{\rm (1)}]  The subspace $\Ker f\subset V_{R_1}$ is totally isotropic. The linear map $f\colon (\Ker f)^\perp\to V_{R_2}$ satisfies 
    \[\omega_{R_1}(x,y)=\omega_{R_2}(f(x),f(y)).\]

    \item[{\rm (2)}] 
For $i=1,2$, let $\pi_i$ denote the projections $H_{R_i}\to V_{R_i}$. 
    Let $I_f\subset H_{R_1}$ denote the inverse image of $(\Ker f)^\perp$ via the projection $\pi_1$. Then the assignment $\psi_f(a,b)=(f(a),b+\Delta(a))$ defines a commutative diagram of groups 
    \begin{equation}\label{exHf}
    \xymatrix{
1\ar[r]&\Ker\psi_f\ar[d]^-{\simeq}\ar[r]&I_f\ar[d]^-{\pi_1}\ar[r]^-{\psi_f}&H_{R_2}\ar[d]^-{\pi_2}\ar[r]&1\\
1\ar[r]&\Ker f\ar[r]&(\Ker f)^\perp\ar[r]^-f& V_{R_2}\ar[r]&1 
}
\end{equation}
with exact rows. Furthermore, the left vertical arrow $\Ker\psi_f\to\Ker f$ and the map 
$\psi_f\colon \Ker\pi_1\to\Ker\pi_2$
are isomorphisms. 
\item[{\rm (3)}] Let the notation be as in {\rm (2)}. Then, for a maximal abelian subgroup $A\subset H_{R_2}$, the inverse image $\psi_f^{-1}(A)$ is maximal abelian in $H_{R_1}$. Moreover,  for any maximal abelian subgroup $A'\subset H_{R_1}$ containing $\Ker\psi_f$, 
there exists a maximal abelian subgroup $A\subset H_{R_2}$ such that $A'=\psi_f^{-1}(A)$. 
\end{itemize}
\end{proposition}
\begin{proof}
(1) Note that for $a,b\in V_{R_i}$ ($i=1,2$),  
\begin{equation}\label{wab}
    \omega_{R_i}(a,b)=f_{R_i}(a+b,a+b)+f_{R_i}(a,a)+f_{R_i}(b,b).
\end{equation} 
Let $c\in\Ker f$. By Lemma~\ref{Kerf} (2), we have $f_{R_1}(c,c)=0$. Together with \eqref{wab}, this shows that $\Ker f$ is totally isotropic with respect to $\omega_{R_1}$. 
From \eqref{das2} and Lemma~\ref{Kerf} (3), we see that for 
$x,y \in (\Ker f)^{\perp}$, 
\[
\omega_{R_1}(x,y)+\omega_{R_2}(f(x),f(y))=\delta(x,y)+\delta(y,x)=0. 
\]
This finishes the proof. 

(2)  First, we check that $\psi_f(a,b)=(f(a),b+\Delta(a)) \in H_{R_2}$. By $a \in (\Ker f)^{\perp}$ and Lemma~\ref{Kerf} (3), we know that $f(a)\in V_{R_2}$. By \eqref{das}, we compute 
\[(b+\Delta(a))^p+(b+\Delta(a))=aR_1(a)+aR_1(a)+f(a)R_2(f(a))=f(a)R_2(f(a)), 
\]
which proves $\psi_f(a,b)\in H_{R_2}$. 

We show that $\psi_f$ is a group homomorphism. We have 
\begin{align*}
\psi_f((a,b) \cdot (a',b'))&=\psi_f(a+a',\ b+b'+f_{R_1}(a,a'))\\
&=(f(a)+f(a'),\ b+b'+\Delta(a+a')+f_{R_1}(a,a')). 
\end{align*}
On the other hand, we compute 
\begin{align*}
\psi_f(a,b) \cdot \psi_f(a',b')&=(f(a),b+\Delta(a)) \cdot (f(a'),b'+\Delta(a'))\\
&=(f(a)+f(a'),\ b+b'+\Delta(a)+\Delta(a')+f_{R_2}(f(a),f(a'))). 
\end{align*}
Since $a' \in (\Ker f)^\perp$, we have $\delta(x,a')=0$ by Lemma~\ref{Kerf} (3).
Hence, by $\delta(a,a')=0$ and \eqref{das2}, we obtain
\[
\psi_f((a,b)\cdot(a',b'))=\psi_f(a,b) \cdot \psi_f(a',b').
\]

From the definition of $\psi_f$, we see that the right square of \eqref{exHf} is commutative. 
Note that $\Ker f\subset(\Ker f)^\perp$, since $\Ker f$ is totally isotropic by (1). Therefore, the diagram \eqref{exHf} is induced. 
By the non-degeneracy of $\omega_{R_1}$ and Lemma~\ref{degcomp}, we have 
\[|(\Ker f)^{\perp}/\Ker f|=|V_{R_1}/\Ker f|/|\Ker f|
=(\deg R_1/\deg f)^2=(\deg R_2)^2=|V_{R_2}|.\] 
Hence, the bottom horizontal line in \eqref{exHf} is exact. 
We consider the commutative diagram of exact sequences 
\[
\xymatrix{
1 \ar[r] & \Ker \pi_1 \ar[r]\ar[d]^{\psi_f}_{\simeq} & I_f \ar[d]^-{\psi_f}\ar[r]^-{\pi_1} & (\Ker f)^{\perp} \ar[r]\ar[d]^f & 1 \\
1 \ar[r] & \Ker \pi_2 \ar[r] & H_{R_2}  \ar[r]^-{\pi_2} & V_{R_2} \ar[r] & 1.  
}
\]
Since $\Delta(0)=0$, the map $\psi_f\colon \Ker \pi_1\to \Ker \pi_2$ is identified with the identity map $\F_p\to \F_p$. Thus the surjectivity of $\psi_f$ follows from the surjectivity of $f \colon (\Ker f)^{\perp} \to V_{R_2}$ and the snake lemma. The bijectivity of the left vertical arrow
in \eqref{exHf} follows from the isomorphism $\psi_f \colon \Ker \pi_1 \xrightarrow{\sim}
\Ker \pi_2$, the surjectivity $\pi_1 \colon I_f \to (\Ker f)^{\perp}$ and 
the snake lemma. 

(3) First, we prove that $A_1\coloneqq \psi_f^{-1}(A)$ is abelian. Take $g,h\in A_1$. Then we have $\psi_f([g,h])=1$. Therefore, we have $[g,h]\in \Ker \psi_f$. This implies that $[g,h]=1$ since the map $\pi_1\colon\Ker\psi_f\to V_{R_1}$ is injective by (2). 

We have $\pi_1(A_1)=f^{-1}(\pi_2(A))$ by 
\eqref{exHf}. 
The maximality of $A$ implies $|\pi_2(A)|=p^{e_2}$. 
From the exactness of the bottom row in \eqref{exHf}, 
it follows that $|\pi_1(A_1)|=\deg f \cdot 
|\pi_2(A)|$, which equals $p^{e_1}$ by 
Lemma~\ref{degcomp}. Thus 
the maximality of $A_1$ follows. 

Conversely, let $A'\subset H_{R_1}$ be a maximal abelian subgroup containing $\Ker\psi_f$. Since $\pi_1(A')$ is totally isotropic and contains $\Ker f$, we must have $A'\subset I_f$. Since $\psi_f$ is a group homomorphism, $\psi_f(A')$ is abelian. Therefore, there exists a maximal abelian subgroup $A\subset H_{R_2}$ such that $A'\subset \psi_f^{-1}(A)$. The equality $A'=\psi^{-1}_f(A)$ then follows from the maximality of $A'$ and $\psi_f^{-1}(A)$. 
\end{proof}

As a consequence, we obtain the following relation between the curves $C_{R_1}$ and $C_{R_2}$. 

 \begin{lemma}\label{phif}
 \begin{itemize}
     \item[{\rm (1)}] There exists a cartesian square of $\F_q$-schemes 
\[
\xymatrix{C_{R_1}\ar[r]^-{\phi_f}\ar[d]&C_{R_2}\ar[d]\\
\A^1_{\F_q}\ar[r]^-f&\A^1_{\F_q}, 
}
\]
where the vertical arrows are given by $(x,y)\mapsto x$ and $\phi_f$ is defined by 
\[\phi_f(x,y) = (f(x),y+\Delta(x)). \]
  \item[{\rm (2)}]  The map $\phi_f$ is finite \'etale.  If $\Ker f \subset \F_q$, then $\phi_f$  is a Galois covering with Galois group 
$\Ker f$. 
 \item[{\rm (3)}] 
The map $\phi_f$ is equivariant with respect to $\psi_f\colon I_f\to H_{R_2}$. 
 \end{itemize}

\end{lemma}
\begin{proof}
(1) Let $C$ denote the fiber product of $\A^1_{\F_q}\xrightarrow{f}\A^1_{\F_q}\leftarrow C_{R_2}$. The $\F_q$-scheme $C$ is identified with the affine curve defined by the equation 
\[y^p+y=f(x)R_2(f(x)).\]
By \eqref{das}, this can be rewritten as 
\[(y+\Delta(x))^p+(y+\Delta(x))=xR_1(x). \]
Therefore, $C$ is isomorphic to $C_{R_1}$ via $(x,y)\mapsto (x,y+\Delta(x))$. The assertion follows. 

(2) The claim follows from (1) and the corresponding assertion for $f\colon \A^1_{\F_q}\to \A^1_{\F_q}$. 

(3) Let $(x,y)\in C_{R_1}$ and $(a,b)\in I_f$. Then the equality
\[
\phi_f((x,y)\cdot(a,b))
=\phi_f(x,y)\cdot\psi_f(a,b)
\]
follows from $\delta(x,a)=0$ in 
Lemma~\ref{das2} (3),
in the same way as the equality
$\psi_f((a,b)\cdot(a',b'))
=\psi_f(a,b)\cdot\psi_f(a',b')$
proved in Proposition~\ref{KerfA} (2).
\end{proof}

\subsection{The case where $R_2(x)=x^p+s_0x$}\label{RS}

Let
$R(x) = \sum_{i=0}^e a_i x^{p^i} \in \mathbb{F}_q[x]$
be an $\F_p$-linearized polynomial with $e\geq1$ and $a_e\neq0$. 
 The following lemma is a characteristic-two analogue of \cite[Lemma~4.9]{Ts}, and it also generalizes \cite[Proposition~9.1]{GV}.
\begin{lemma}\label{q0}
Let $(a,b)\in H_R$ satisfy $a\neq 0$ and $f_R(a,a)=0$.
Let $f(x)\coloneqq x^p+a^{p-1} x$ and 
$\Delta_b(x)\coloneqq (x/a)(b(x/a) + f_R(x,a))$. Assume $(a,b)\in \F_q^2$. 
\begin{itemize}
\item[{\rm (1)}] There exists an $\F_p$-linearized polynomial $P_{b}(x) \in 
\mathscr{A}$ of degree $p^{e-1}$ such that 
\[
\Delta_b(x)^p+\Delta_b(x)=x R(x)+f(x) P_{b}(f(x)). 
\]
\item[{\rm (2)}] 
The morphism
\[
C_R \to C_{P_{b}}, \quad 
(x,y) \mapsto (f(x),\,y+\Delta_b(x))
\]
is a finite \'etale Galois covering with Galois group 
 $\Ker f=\F_p\, a$. 
\end{itemize}
\end{lemma}
\begin{proof}
(1) 
Let $x_1\coloneqq x/a$. Then 
$a^{-p}f(x)=x_1^p+x_1$. 
A direct computation shows that
\begin{align*}
xR(x)+\Delta_b(x)^p+\Delta_b(x)&=
xR(x)
+b^p x_1^{2p}+b x_1^2+x_1^p
{f_R(x,a)}^p+x_1 f_R(x,a) \\
&=xR(x)+b(x_1^{2p}+x_1^2)+
a^{-2p+1} R(a) x^{2p}\\
&\ \ \ +a^{-p} f(x) f_R(x,a)+(x/a)^p(a R(x)+x R(a)) \\
&=a^{-p}f(x) (a R(x)+a^{-p+1}
R(a) x^p+b (x_1^p+x_1)+f_R(x,a)), 
\end{align*}
where we use $b^p+b=a R(a)$ and \eqref{a} for the second equality. 
Let
\[ 
P(x) \coloneqq a^{-p}(a R(x)+a^{-p+1}
R(a) x^p+b (x_1^p+x_1)+f_R(x,a)) 
\in \mathscr{A}. \] 
By $P(x) \in \mathscr{A}$ and $f_R(a,a)=0$, 
we have $P(\xi a)=\xi P(a)=
\xi a^{-p}f_R(a,a)=0$ for $\xi \in \F_p$.  
Thus there exists an $\F_p$-linearized polynomial
$P_{b}(x) \in \mathscr{A}$ such that 
$P(x)=P_{b}(f(x))$. Hence we obtain the displayed equality. The assertion $\deg P_b=p^{e-1}$ follows from Lemma~\ref{degcomp}.

(2) This follows from Lemma~\ref{phif} for $(R_1,R_2,\Delta)=(R,P_b,\Delta_b)$.
\end{proof}
In the rest of this section, 
let $\overline{A} \subset V_R$ be a maximal 
totally isotropic subspace with respect to $\omega_R$
and let $A\coloneqq \pi^{-1}(\overline{A}) \subset H_R$. 
Assume $A \subset \F_q^2$. 
Put  
\begin{equation}\label{oUdef}
    \oU\coloneqq \{a \in \overline{A} \mid 
f_R(a,a)=0\}.
\end{equation} 

\begin{lemma}\label{oU}
    The subset $\oU\subset V_R$ is an $\F_p$-vector space. We have 
    \[e-1\leq\dim\oU\leq e. \]
\end{lemma}

\begin{proof}
The map $\overline{A}\to \F$ defined by $a\mapsto f_R(a,a)$ is additive since 
\[
f_R(a+a',a+a')-f_R(a,a)-f_R(a',a')=f_R(a,a')+f_R(a',a)=\omega_R(a,a')=0
\]
for $a,a'\in\overline{A}$. 
Thus the first assertion follows since 
$f_R(ba,ba)=b^2f_R(a,a)=0$ for $a \in \oU$ and $b\in\F_p$. 

From \eqref{a}, it follows that  
\[f_R(x,x)^p+f_R(x,x)=x^{p^e}E_R(x).\]
  Evaluating at $x = a \in \overline{A}$, it follows that $f_R(a,a) \in \F_p$. The second assertion then follows since 
   $\overline{A}$ has dimension $e$ and $\oU$ is the kernel of $\overline{A}\to\F_p,\ a \mapsto f_R(a,a)$. 
\end{proof}

In the lemmas below, we actually prove $\dim\oU=e-1$. 

\begin{lemma}\label{q01}
Assume $e=1$. Then we have $\oU=0$. 
 \end{lemma}
 \begin{proof}
 Let $a\in \oU$. 
Since $e=1$,  \eqref{-a} is reduced to $f_R(x,y)=a_1 x y^p$, where $a_1 \neq 0$ is the coefficient of $x^p$ in $R(x)$. Substituting $x = y = a$ gives 
$f_R(a,a) = a_1 a^{p+1}=0$.
Hence we have $a=0$ since $a_1 \neq 0$. 
 \end{proof}
\begin{lemma}\label{suc}
\begin{itemize}
    \item [{\rm (1)}]

 There exist a separable $\F_p$-linearized polynomial 
$f_{\oU}(x)\in \mathscr{A}$,
an $\F_p$-linearized polynomial $S(x) = s_1 x^p + s_0 x\ (s_1 \neq 0) \in \mathscr{A}$,  
and a polynomial $\Delta(x) \in 
\F_q[x]$ such that 
$\oU=\Ker f_{\oU}$ and 
\[
\Delta(x)^p+\Delta(x)=xR(x)+f_{\oU}(x)
S(f_{\oU}(x)). 
\] 

\item[{\rm (2)}]
We have $\dim \oU=e-1$. 
\end{itemize}
\end{lemma}

\begin{proof}
We prove the claims (1) and (2) simultaneously by induction on $e$.
Assume $e=1$. 
By Lemma~\ref{q01}, we have 
$\oU=0$. By taking 
$f_{\oU}(x) =x$, $S(x) = R(x)$, and  $\Delta(x) = 0$, the assertion follows. 

We assume that $e\geq2$. Then, by Lemma~\ref{oU}, there exists a non-zero element $a\in\oU$. Lift $a$ to $(a,b)\in H_R$. Then, by $a \neq 0$, $f_R(a,a)=0$ and Lemma~\ref{q0}, we can find an $\F_p$-linearized polynomial $P_b(x)\in\mathscr{A}$ of degree $p^{e-1}$ and a polynomial $\Delta_b(x)\in \F_q[x]$ such that 
 \[\Delta_b(x)^p+\Delta_b(x)=xR(x)+f(x)P_b(f(x)),\]
where $f(x)=x^p+a^{p-1}x$. Consider the surjection in Proposition~\ref{KerfA}:
\[
f\colon (\Ker f)^\perp \to V_{P_b}.
\]
Since $\overline{A} \subset (\Ker f)^\perp$, we have $f(\overline{A}) \subset V_{P_b}$. 
Moreover, by Proposition~\ref{KerfA}~(1) and $|f(\overline{A})| = p^{e-1}$, 
the image $f(\overline{A})$ is maximal totally isotropic with respect to $\omega_{P_b}$.

Furthermore, the image $f(\overline{U})$ is equal to the kernel $K$ of the homomorphism \[f(\overline{A})\to \F_p,\ x \mapsto f_{P_b}(x,x).\]
Indeed, by Lemma~\ref{Kerf} (3), the composite map 
\[
\overline{A}\xrightarrow{f}f(\overline{A})\to \F_p
\]
is equal to $x\mapsto f_R(x,x)$, which shows that $\overline{U}=f^{-1}(K)$. Since the map $\overline{A}\to f(\overline{A})$ is surjective, it follows that $f(\overline{U})=K$. By the induction hypothesis, we have 
$\dim K=e-2$ and hence $\dim \oU=\dim K+1=e-1$.

By the induction hypothesis, we can write 
\[
\Delta_1(x)^p+\Delta_1(x)=xP_b(x)+g(x)S(g(x)), 
\]
where $S(x)$ is of the form $s_1x^p+s_0x$ with $s_1 \neq 0$, and $\Ker g=f(\oU)$. Then we may take $\Delta(x)=\Delta_1(f(x))+\Delta_b(x)$ and $f_{\oU}(x)=g(f(x))$. Since $f(x)$ and $g(x)$ are separable, so is $f_{\oU}(x)$. 
Since $g$ is separable with $\dim\Ker g=e-2$, we have $\deg g=p^{e-2}$, and hence
$\deg f_{\oU}=p^{e-1}$.
Since $\Ker g=f(\oU)$, we have $\oU \subset \Ker f_{\oU}$. 
Since $\dim \oU=e-1$, we conclude that $\oU=\Ker f_{\oU}$. This completes the proof. 
\end{proof}
The statement of Lemma~\ref{suc} (1) can be strengthened as follows. 
\begin{proposition}\label{choice}
Assume $A\subset \F_q^2$. Let $\oU$ be as in \eqref{oUdef}. 
Then there exists a separable $\F_p$-linearized polynomial 
$f_{\oU}(x)\in\mathscr A$ such that the following hold:
\begin{itemize}
\item $\oU=\Ker f_{\oU}$ and 
$\overline{A}=\Ker (f_{\oU}^p+f_{\oU})$;
\item there exist $S(x)=x^p+s_0x\in\mathscr A$ and a polynomial 
$\Delta(x)\in\F_q[x]$ such that
\begin{equation}\label{del}
\Delta(x)^p+\Delta(x)
=xR(x)+f_{\oU}(x)S(f_{\oU}(x)).
\end{equation}
\end{itemize}
Moreover, the morphism
\begin{equation}\label{rr}
\phi_{R,S}\colon C_R \to C_S,\quad
(x,y)\mapsto (f_{\oU}(x),\,y+\Delta(x))
\end{equation}
is a finite \'etale Galois cover with Galois group $\oU$.
\end{proposition}
\begin{proof}
By Lemma~\ref{suc}, we can find $f_{\oU}(x)$, $\Delta(x)$, and $S(x)=s_1x^p+s_0x$ satisfying 
\begin{equation}\label{delb}
\Delta(x)^p+\Delta(x)=xR(x)+f_{\oU}(x)S(f_{\oU}(x)).
\end{equation}
Since $\dim \oU=e-1$ by Lemma~\ref{suc} (2), the image $f_{\oU}(\overline{A})$ is one-dimensional. Therefore, we can write 
$f_{\oU}(\overline{A})=\F_p \eta$ with 
$\eta \in \F_q^{\times}$. Then $f_{\oU}(a)^p+\eta^{p-1} f_{\oU}(a)=0$ for any $a \in \overline{A}$.  
Dividing this equality by $\eta^{-p}$ and 
replacing $\eta^{-1}f_{\oU}(x)$ by $f_{\oU}(x)$, we have $f_{\oU}(a)^p+f_{\oU}(a)=0$ for any $a \in \overline{A}$. Since $f_{\oU}(x)^p+f_{\oU}(x)$ has degree $p^e$ and $|\overline{A}|=p^e$, 
we obtain $\overline{A}=\Ker (f_{\oU}^p+f_{\oU})$.  If we replace 
 $S(x)$ by $\eta S(\eta x)$, the equality 
\eqref{delb} is preserved. 

Write $S(x)= s_1 x^p + s_0 x$. By Lemma~\ref{Kerf} (3), the image 
$f_{\oU}(\overline{A})=\F_p$ is contained in $V_S$. 
In particular, $1 \in V_S$. 
Since we have $E_{S}(x)=s_1^p x^{p^2}+s_1x$, this implies that  
 $s_1^p+s_1=E_{S}(1)=0$. Thus we have $s_1 \in \F_p$. 
We set $S_1(x) \coloneqq x^p + s_0 s_1^{-1} x$. Since $s_1^{1/2}\in\F_p$,  we have 
$s_1^{1/2} x S_1(s_1^{1/2} x) = x S(x)$. Consequently, by replacing $S(x)$ with $S_1(x)$
and $f_{\oU}(x)$ with $s_1^{1/2}f_{\oU}(x)$, we may normalize so that $s_1=1$.

The last claim follows from Lemma~\ref{phif} by applying it with $(R_1,R_2, f)=(R,S,f_{\oU})$. 
\end{proof}

\begin{lemma}\label{kk}
Let the notation and assumptions be as in Proposition~\ref{choice}. 
\begin{itemize}
    \item[{\rm (1)}] There exists $\a \in \F_q$ such that 
$\alpha^{p^{-1}} + \alpha = s_0 + 1$. 
\item[{\rm (2)}] Let $A_S$ denote the inverse image $\pi_S^{-1}(\F_p)$ of $\F_p\subset \F_{p^2}$ via the projection $\pi_S\colon H_S\to V_S=\F_{p^2},\ (a,b) \mapsto a$. Then 
the upper horizontal line in \eqref{exHf} restricts to a split short exact sequence of abelian groups 
\[0\to \oU\to A\to A_S\to 0.\] 
\end{itemize}
\end{lemma}
\begin{proof}
(1) 
The first condition in Proposition~\ref{choice} implies a surjection $\overline{A} \to \F_p,\ 
x \mapsto f_{\oU}(x)$. 
Thus, we can take $a \in \overline{A}$ 
such that $f_{\oU}(a)=1$.  
Take a lift $(a,b) \in A$ of $a$. 
Then we have 
\[
\Delta(a)^p+\Delta(a)=a R(a)+
S(1)=b^p+b+s_0+1, 
\]
and hence $(\Delta(a)+b)^p+(\Delta(a)+b) = s_0+1$. 
Since $A \subset \F_q^2$ and $\Delta(x) \in \F_q[x]$, we have $\Delta(a)+b \in \F_q$, 
so we may take $\alpha =(\Delta(a)+b)^p$.

(2) By \eqref{del} and Proposition~\ref{KerfA} (2), we have an exact sequence 
\[1\to \oU\to I_{f_{\oU}}\xrightarrow{\psi_{f_{\oU}}} H_S\to 1,\]
where the map $\oU\to I_{f_{\oU}}$ is given by $a\mapsto (a,\Delta(a))$. This restricts to an exact sequence 
\[0\to \oU\to A\to \psi_{f_{\oU}}(A)\to 0. 
\]
Since $A=\pi^{-1}(\overline{A})$ and 
$\Ker f_{\oU}=\oU \subset \overline{A}$, 
Proposition~\ref{KerfA} (2) implies that 
\[\psi_{f_{\oU}}(A)=\psi_{f_{\oU}}(\pi^{-1}(\overline{A}))=\pi_S^{-1}(f_{\oU}(\overline{A}))=\pi_S^{-1}(\F_p)=A_S. 
\]
 Therefore, we obtain an exact sequence in (2). It splits, since $A$ is a $\Z/4\Z$-module by Lemma~\ref{basic} (4), and
$A_S$ is isomorphic to $W_2(\F_p)$ by Lemma~\ref{maximalabeliansubgroupwitt},
which is $\Z/4\Z$-free by Lemma~\ref{free}.
\end{proof}

By Proposition~\ref{KerfA} (2), the exact sequence in Lemma~\ref{kk} (2) fits into a commutative diagram of exact sequences 
\begin{equation}\label{exRS}
    \xymatrix{
0\ar[r]&\oU\ar[d]^-{{\rm id}_{\oU}}\ar[r]&A\ar[d]^-{\pi}\ar[r]&A_S\ar[d]^-{\pi_S}\ar[r]&0 \\
0\ar[r]&\oU\ar[r]^-{i}&\overline{A}\ar[r]^-{f_{\oU}}& \F_p\ar[r]&0, 
}
\end{equation}
Let 
\begin{equation}\label{retAU}
      r \colon \overline{A}\to \oU
\end{equation}
be an $\F_p$-linear map satisfying $r \circ i={\rm id}_{\oU}$. 
Fix $\alpha\in \F_q$ as in Lemma~\ref{kk} (1). 
Then the homomorphisms $ r \circ \pi \colon A\to \oU$ 
and $f_{\alpha^{p^{-1}}}$ in \eqref{aw} 
together induce an isomorphism 
\begin{equation}\label{AUW}
   A\xrightarrow{\sim}  W_2(\F_p)\times \oU.  
\end{equation}

\section{Proof of the main theorem}
\label{Section:MainTheorem}

In this section, we calculate the Frobenius eigenvalues
and the $L$-polynomials of van der Geer--van der Vlugt curves
associated with more general $\F_p$-linearized polynomials.

\subsection{A lemma on smooth $\overline{\Q}_{\ell}$-sheaves}

We use the following elementary fact on smooth $\overline{\Q}_{\ell}$-sheaves. 
Let $k$ be a field. Let $H$ be a finite group acting freely on an affine variety $X$ over $k$. Then the natural map $X\to X/H$ is a finite \'etale Galois covering with Galois group $H$. For an $H$-representation $\xi$ on a finite-dimensional $\overline{\Q}_\ell$-vector space, we write $\mathcal{L}_\xi$ for the smooth $\overline{\Q}_\ell$-sheaf on $X/H$ defined by the finite \'etale $H$-covering $X\to X/H$ and $\xi$. 

For representations $\xi_i$ of groups $G_i$ ($i = 1, 2$),
we write $\xi_1 \boxtimes \xi_2$ for the external tensor product representation of $G_1 \times G_2$ associated to $\xi_1$ and $\xi_2$.

\begin{lemma}\label{ele}
Let $X$ be an affine variety over a field $k$. 
Let $G_1$ and $G_2$ be finite groups and let $G \coloneqq G_1 \times G_2$.
Assume that $G$ acts on $X$ freely. 
For $i=1,2$, 
let $\xi_i$ be a finite-dimensional irreducible representation of $G_i$ over $\overline{\Q}_{\ell}$. 
Let $\mathcal{L}_{\xi_1}$, $\mathcal{L}_{\xi_2}$, and $\mathcal{L}_{\xi_1\boxtimes\xi_2}$  be  the $\overline{\Q}_{\ell}$-sheaves  
on $X/G$ defined by the coverings $X/G_2 \to X/G$, $X/G_1\to X/G$, and $X\to X/G$, respectively. 
Then we have an isomorphism of $\overline{\Q}_{\ell}$-sheaves: 
\[
\mathcal{L}_{\xi_1 \boxtimes \xi_2}
\simeq \mathcal{L}_{\xi_1} \otimes \mathcal{L}_{\xi_2}.
\]
\end{lemma}
\begin{proof}
We consider the following cartesian diagram 
\[
\xymatrix{
X \ar[r]^-{p_2}\ar[d]^{\pi_1} & X/G_2 \ar[d]^{\pi_2}\\
X/G_1 \ar[r]^-{p_1} & X/G. 
}
\]
Let $\mathcal{L}'_{\xi_i}$ denote the $\overline{\Q}_{\ell}$-sheaf on 
$X/G_i$ defined by $\xi_i$ and 
the finite \'etale $G_i$-covering 
$X \to X/G_i$. 
Let $\mathrm{Irr}(G_i)$ denote the set of 
isomorphism classes of irreducible representations of $G_i$ over $\overline{\Q}_{\ell}$. 
By the proper base change theorem, we have 
\[
\bigoplus_{\xi' \in \mathrm{Irr}(G_1)} 
\mathcal{L}'_{\xi'} \simeq {\pi_1}_{\ast} \overline{\Q}_{\ell}
\simeq {\pi_1}_{\ast}p_2^\ast \overline{\Q}_{\ell} \simeq p_1^\ast {\pi_2}_{\ast} \overline{\Q}_{\ell}\simeq 
\bigoplus_{\xi' \in \mathrm{Irr}(G_1)} 
p_1^\ast \mathcal{L}_{\xi'}. 
\] 
By taking the $\xi_1$-isotypic component, 
we obtain $\mathcal{L}'_{\xi_1} \simeq p_1^\ast \mathcal{L}_{\xi_1}$. 
Applying the projection formula, we find: 
\[
\bigoplus_{\xi'_2 \in 
\mathrm{Irr}(G_2)} \mathcal{L}_{\xi_1 \boxtimes \xi'_2}
\simeq p_{1\ast} \mathcal{L}'_{\xi_1}
\simeq p_{1\ast}p_1^{\ast} \mathcal{L}_{\xi_1} \\
 \simeq \mathcal{L}_{\xi_1} \otimes p_{1\ast}
\overline{\Q}_{\ell} \simeq 
\bigoplus_{\xi'_2 \in \mathrm{Irr}(G_2)} \mathcal{L}_{\xi_1} \otimes \mathcal{L}_{\xi'_2}. 
\]
Taking the $\xi_1 \boxtimes \xi_2$-isotypic component on both sides, 
we obtain the claim. 
\end{proof}

\subsection{Calculation of the Frobenius eigenvalues}

We use the same notation as in Sections \ref{Section:Geometry} 
and \ref{Section:Quotients}.

Let $R(x) \in \mathscr{A}$ be an $\F_p$-linearized polynomial of degree $p^e$ with $e \ge 1$.
Let $A \subset H_R$ be a maximal abelian subgroup. We assume $A \subset \F_q^2$.

As in Proposition~\ref{choice}, we take a separable $\F_p$-linearized polynomial $f_{\oU}(x) \in \mathscr{A}$. Then $F_A(x) \coloneqq f_{\oU}(x)^p+f_{\oU}(x)$ satisfies Lemma~\ref{F_A}. 

\begin{lemma}\label{aF=Fa}
There exists an $\F_p$-linearized polynomial $a(x)\in\mathscr{A}$ such that
\[
x^q+x = a(F_A(x)) = F_A(a(x)).
\]
The polynomial $a(x)$ induces a surjective map
$\F_q \to \overline{A},\ x \mapsto a(x)$.
\end{lemma}
\begin{proof}
Since  $\overline{A} \subset \F_q$ and $F_A(x)$ is 
separable, there exists a polynomial $a(x) \in \mathscr{A} \setminus \{0\}$ such that $x^q+x=a(F_A(x))$. By replacing $x$ with $a(x)$, we obtain 
\[a(x)^q+a(x)=a(F_A(a(x))).\]
Since all coefficients of $a(x)$ lie in $\F_q$ and $a(x)$ is additive, we have
\[
a(x^q + x - F_A(a(x))) = 0.
\]
Because $a(x)$ is nonzero and $x^q + x - F_A(a(x)) \in \mathscr{A}$, 
it follows that
$x^q + x = F_A(a(x)).$

For $x\in \F_q$, we have $F_A(a(x))=x^q+x=0$, which shows $a(\F_q)\subset \overline{A}$. On the other hand, for any $y\in \overline{A}$, there exists $z\in\F$ such that $a(z)=y$. Then 
$z \in \F_q$ since 
\[z^q+z=F_A(a(z))=F_A(y)=0,\]
which shows that the map $\F_q\to\overline{A},\ x \mapsto a(x)$ is surjective. 
\end{proof}
Fix $\alpha\in \F_q$ as in Lemma~\ref{kk} (1) and $r\colon \overline{A}\to \oU$ as in \eqref{retAU}. These induce an isomorphism 
 $\varphi \colon A \xrightarrow{\sim} W_2(\mathbb{F}_p) \times 
\oU$ as in \eqref{AUW}.  Let $\psi\in \F_p^\vee \backslash \{ 1 \}$ be a nontrivial character, and take $\xi \in A_{\psi}^{\vee}$. Using the isomorphism $\varphi$, we decompose $\xi=\xi_{W_2}\boxtimes\xi_{\oU}$, where $\xi_{W_2}\in W_2(\F_p)^\vee$ and $\xi_{\oU}\in \oU^\vee$. 
We consider 
the character 
\[
\F_q \xrightarrow{a} \overline{A} \xrightarrow{r} \oU \xrightarrow{\xi_{\oU}}
\overline{\Q}_{\ell}^{\times}. 
\]
Let $\beta_{\xi} \in \F_q$ denote the 
element satisfying 
\begin{equation}\label{bxi}
\xi_{\oU} \circ r
\circ a(x)=\psi \circ \Tr_{q/p}(\beta_{\xi} x) \quad 
\textrm{for any $x \in \F_q$}. 
\end{equation}
We identify $W_2(\F_p)$ with a subgroup of 
$A$ via the 
isomorphism $\varphi \colon A \xrightarrow{\sim}
W_2(\F_p) \times \oU$. 
Then taking quotients induces a  
finite \'etale Galois covering 
\[
\pi_{\oU} \colon 
C_R/W_2(\F_p) \to C_R/A \simeq 
\A^1_{\F_q}
\]
with 
Galois group $\oU$. Let $\mathcal{L}_{\xi_{\oU}}$
denote the $\overline{\Q}_{\ell}$-sheaf on $\A^1_{\F_q}$
defined by the covering $\pi_{\oU}$ and $\xi_{\oU}$.

\begin{lemma}\label{sps}
We have $\mathcal{L}_{\xi_{\oU}}
\simeq \mathcal{L}_{\psi}(\beta_{\xi} s)$. 
\end{lemma}
\begin{proof}
We recall that the quotient morphism $C_R \to C_R/A$ is identified with the composite map 
\[
C_R \xrightarrow{(x,y) \mapsto x} \A_{\F_q}^1 \xrightarrow{F_A} \A_{\F_q}^1, 
\]
where the first morphism corresponds to  
$C_R \to C_R/Z(H_R)$. Let $\pi \colon A\to\overline{A}$ be the projection. 
Since $r \circ \pi$ is trivial on the subgroup $Z(H_R)\subset A$, it follows that 
$Z(H_R) $ is contained in the subgroup $W_2(\F_p) \subset A$ and $W_2(\F_p)/Z(H_R)=\Ker r$. 
Thus $F_A \colon \A_{\F_q}^1 \to \A_{\F_q}^1$ factors as 
\[
\A_{\F_q}^1 \simeq C_R/Z(H_R) \to 
C_R/W_2(\F_p) \xrightarrow{\pi_{\oU}} 
C_R/A\simeq \A^1_{\F_q}. 
\]
Therefore, by Lemma~\ref{aF=Fa}, we have a commutative diagram 
\[
\xymatrix{\A^1_{\F_q}\ar[d]_-a\ar[rrd]^-{x^q+x}&&\\
\A_{\F_q}^1 \ar[r]\ar@/_25pt/[rr]^{F_A} &
C_R/W_2(\F_p) \ar[r]_-{\pi_{\oU}}&
 \A^1_{\F_q}. 
}
\]
Then the assertion follows from \eqref{bxi} since the smooth sheaf on $\A^1_{\F_q}$ induced from the covering $x^q+x=s$ and the character $\F_q \ni s \mapsto \psi(\Tr_{q/p}(\beta_{\xi}s))$ is equal to $\mathcal{L}_{\psi}(\beta_{\xi}s)$.
\end{proof}

\begin{proposition}\label{qpp}
We have an isomorphism of $\overline{\Q}_{\ell}$-sheaves on
$\A_{\F_q}^1$:
\[ \mathcal{Q}_{\xi} \simeq \mathcal{L}_{\xi_{W_2}}(s, \a s^2 + \beta_{\xi} s). \] 
\end{proposition}

\begin{proof}
The decomposition  $A=W_2(\F_p)\times\oU$ induces a cartesian square 
\[
\xymatrix{
C_R\ar[d]\ar[r]&C_R/\oU\ar[d]\\
C_{R}/W_2(\F_p)\ar[r]&C_R/A=\A^1_{\F_q}. 
}
\]
By $F_A=f_{\oU}^p+f_{\oU}$ and Proposition~\ref{choice}, the right vertical arrow is identified with $C_S\to\A^1_{\F_q},\ (x,y)\mapsto x^p+x$. Since $\xi_{W_2}|_{Z(H_R)} = \psi$, 
we have an isomorphism 
$\mathcal{L}_{\psi}(\beta_{\xi} s) \simeq 
\mathcal{L}_{\xi_{W_2}}(0,\beta_{\xi} s)$. 
Thus from Lemmas~\ref{QL}, \ref{ele}, and~\ref{sps}, it follows that 
\begin{align*}
\mathcal{Q}_{\xi} \simeq \mathcal{L}_{\xi_{W_2}}(s, \alpha s^2) \otimes \mathcal{L}_{\xi_{\oU}} 
&\simeq \mathcal{L}_{\xi_{W_2}}(s,\a s^2)
\otimes \mathcal{L}_{\xi_{W_2}}(0,\beta_{\xi} s) \\
& \simeq 
\mathcal{L}_{\xi_{W_2}}((s,\a s^2)
+(0,\beta_{\xi} s))
\simeq 
\mathcal{L}_{\xi_{W_2}}(s, \a s^2 + \beta_{\xi} s). 
\end{align*}
\end{proof}
Recall that we fix a faithful character $\xi_2 \in W_2(\mathbb{F}_2)^{\vee}$ 
and set
\[
  \xi_q \coloneqq \xi_2 \circ \Tr_{q/2} \in W_2(\F_q)^{\vee}
  \quad \text{and} \quad
  \sqrt{-1} \coloneqq \xi_2(1,0) \in \overline{\Q}_{\ell}^{\times}.
\]
As observed in the proof of Lemma~\ref{sps}, we have
$Z(H_R)\subset W_2(\F_p)$.
Since $Z(H_R)\simeq \F_p$ and $W_2(\F_p)$ is a free $\Z/4\Z$-module by Lemma \ref{free}, it follows that
$Z(H_R)=W_2(\F_p)[2]=\{0\}\times \F_p$.
Since the restriction of
$\xi=\xi_{W_2}\boxtimes\xi_{\oU}$ to $Z(H_R)=2 W_2(\F_p)$ is nontrivial by assumption, 
$\xi_{W_2}$ has order $4$.
By Lemma~\ref{trw}, there exists a unique element $(a_{W_2},b_{W_2}) \in W_2(\F_p)$ such that $a_{W_2} \neq 0$ and 
$\xi_{W_2} = \xi_2 \circ \Tr_{(a_{W_2},b_{W_2})}$.  
Let
\[ c_{R,\xi} \coloneqq (\alpha+(b_{W_2}/a_{W_2}^2))^{1/2} + a_{W_2} \beta_{\xi} \in \F_q. \] 
We have $(c_{R,\xi},0) \in W_2(\F_q)$.

The following result is the main theorem of this paper.

\begin{theorem}\label{qppc}
We have
\[
  \tau_{R,\xi,q} = \xi_{q}(c_{R,\xi},0)^{-1} \cdot (-1-\sqrt{-1})^{[\F_q : \F_2]}. 
\]
\end{theorem}

\begin{proof}
The assertion follows from Lemmas~\ref{HDGr}, \ref{plf} (2),  
 and Proposition~\ref{qpp}.  
\end{proof}

\begin{corollary}\label{qaq2}
Further assume that 
$H_R \subset \mathbb{F}_q^2$. 
\begin{itemize}
    \item[{\rm (1)}] The geometric Frobenius element ${\rm Fr}_q$ acting on $H^1(\overline{C}_{R,\F},\overline{\Q}_{\ell})$ has a unique eigenvalue 
\[ \xi_q(\a,0)^{-1} \cdot (-1-\sqrt{-1})^{[\F_q : \F_2]}. \]  
\item[{\rm (2)}] The degree $[\F_q : \F_2]$ is even. 
Let $d\coloneqq [\F_q:\F_2]/2 \in \mathbb{Z}$. 
Then 
\[
\Tr_{q/2}(\alpha)\equiv d \pmod 2. 
\] 
\end{itemize}
\end{corollary}
\begin{proof}
(1) We take  $\xi=\xi_{W_2}\boxtimes\xi_{\oU} \in A_{\psi}^{\vee}$ such that 
$a_{W_2}\neq0$, $b_{W_2}=0$, and $\xi_{\oU}=1$. Then $c_{R,\xi}=\a^{1/2}$. 
Since 
$H_R \subset \F_q^2$,  Lemma~\ref{Stone} (3) implies that 
$\mathrm{Fr}_q$ acts on $H^1(\overline{C}_{R,\F},\overline{\Q}_{\ell}) \simeq 
 H_{\rm c}^1(C_{R,\F},\overline{\Q}_{\ell})$
 as the scalar multiplication by $\tau_{R,\xi,q}$.
 The claim follows from
 Theorem~\ref{qppc} and $\xi_q(\a,0)=\xi_q(\a^{1/2},0)$. 

 (2) The algebraic integer $ \xi_q(\a,0)^{-1} \cdot (-1-\sqrt{-1})^{[\F_q : \F_2]}$ is a rational number, hence an integer. Indeed, by (1) and Proposition~\ref{14} (1), we have 
\[
\Tr({\rm Fr}_q \mid H^1(\overline{C}_{R,\F},\overline{\Q}_\ell))=2g\cdot \xi_q(\a,0)^{-1} \cdot (-1-\sqrt{-1})^{[\F_q : \F_2]}, 
\]
where $g$ is the genus of $\overline{C}_R$. Since the trace on the left-hand side and $g$ are integers,
we have $\xi_q(\a,0)^{-1} \cdot (-1-\sqrt{-1})^{[\F_q : \F_2]} \in \Q$.

Let $\sqrt{2} \in \overline{\Q}_{\ell}$ be a square root of $2$.
Then $\zeta_8 \coloneqq \tfrac{-1-\sqrt{-1}}{\sqrt{2}}$ is a primitive $8$-th root of unity, since $\zeta_8^2=\sqrt{-1}=\xi_2(1,0)$. We may write 
\[\xi_q(\a,0)^{-1} \cdot (-1-\sqrt{-1})^{[\F_q : \F_2]}=\xi_q(\a,0)^{-1} \cdot \zeta_8^{[\F_q : \F_2]}\cdot 2^{[\F_q:\F_2]/2}. \]
Recall that $\xi_q(\a,0)$ is a $4$-th root of unity. 
If $[\F_q:\F_2]$ is odd, then $\xi_q(\a,0)^{-1} \cdot \zeta_8^{[\F_q : \F_2]}$ is a primitive $8$-th root of unity, which is impossible since the left-hand side lies in $\Q$.
Therefore, $[\F_q:\F_2]$ must be even. 
Let $d=[\F_q:\F_2]/2 \in \Z$ and let 
$\bar{d}\coloneqq d \pmod 2 \in \F_2$. 
Since $\zeta_8^2=\xi_2(1,0)$, we can write 
\[\xi_q(\a,0)^{-1} \cdot (-1-\sqrt{-1})^{2d}=\xi_2(\Tr_{q/2}(\alpha)-\bar{d}, \ast)^{-1}\cdot 2^d \]
for some $\ast \in \F_2$.
This is an integer if and only if $\Tr_{q/2}(\alpha)-\bar{d}=0$. 
\end{proof}

\section{Maximal curves of the form $y^p+y=x^{p+1}+a_0x^2$}
\label{Section:maximalcurve}

In this section, as an illustration of Theorem~\ref{hc},
we give $\F_{p^4}$-maximal
van der Geer--van der Vlugt curves
associated with $\F_p$-linearized polynomials of the form $x^p + a_0 x$. By taking finite quotients of these curves, we obtain more examples of maximal curves. 

\subsection{Maximal and minimal curves}
\label{Subsection:maximalcurve}

Here we briefly recall the notation on maximal and minimal curves.

Let $C$ be a smooth projective geometrically connected curve over $\F_q$ of genus $g(C)$.
By the Hasse--Weil inequality,
we have
\[
  \bigl||C(\F_q)| - (q + 1)\bigr| \leq 2 g(C) \sqrt{q}.
\]

The curve $C$ is said to be \emph{$\F_q$-maximal} if
$|C(\F_q)| = q + 1 + 2 g(C) \sqrt{q}$. 
Equivalently, its $L$-polynomial is
\[
L_{C/\F_q}(T) = \det\bigl(1-\mathrm{Fr}_q \, T \mid H^1(C_{\overline{\F}_q},\overline{\Q}_{\ell})\bigr)
= (1 + \sqrt{q}\, T)^{2 g(C)}.
\]
In other words, $\mathrm{Fr}_q$ acts on 
$H^1(C_{\overline{\F}_q},\overline{\Q}_{\ell})$ 
as scalar multiplication by $-\sqrt{q}$.

Similarly, $C$ is \emph{$\F_q$-minimal} if
$|C(\F_q)| = q + 1 - 2 g(C) \sqrt{q}$, 
equivalently
\[
L_{C/\F_q}(T) = (1 - \sqrt{q}\, T)^{2g(C)},
\]
which means that $\mathrm{Fr}_q$ acts as scalar multiplication by $\sqrt{q}$ on 
$H^1(C_{\overline{\F}_q},\overline{\Q}_{\ell})$.

Constructing explicit maximal (or minimal) curves is an interesting problem in number theory,
with potential applications in coding theory; 
see \cite{St} for more details. 

\subsection{Construction of maximal curves}
We write $p=2^{f_0}$. 
\begin{lemma}\label{xip4}
\begin{itemize}
\item[{\rm (1)}] 
We have
\[
\F_{p^2} \subset \Ker \Tr_{p^4/p}
= \{ x^p + x \mid x \in \F_{p^4} \}.
\]

\item[{\rm (2)}]
Let $a \in \F_{p^2}$ and let $\alpha \in \F_{p^4}$ satisfy
$\alpha^p + \alpha = a + 1$.

\begin{itemize}
\item[{\rm (i)}]
Then
\[
\Tr_{p^2/p}(\alpha^{p^2+1}) = (a+1)^{p+1}.
\]

\item[{\rm (ii)}]
Moreover,
\[
\xi_{p^4}(\alpha,0)
= (-1)^{\Tr_{p/2}(a^{p+1}) + f_0}.
\]
\end{itemize}

\end{itemize}
\end{lemma}

\begin{proof}
(1) By Lemma~\ref{Hilbert90}, we have
\[
\Ker\Tr_{p^4/p}
=
\{x^p-x \mid x\in \F_{p^4}\}
=
\{x^p+x \mid x\in\F_{p^4}\},
\]
since the characteristic is $2$.
Moreover, for $y\in\F_{p^2}$, we have 
$\Tr_{p^4/p}(y)=2 \Tr_{p^2/p}(y)=0$, and hence $\F_{p^2}\subset\Ker\Tr_{p^4/p}$.

(2) (i) We have 
\[(a+1)^{p+1}=(\alpha+\alpha^p)(\alpha^p+\alpha^{p^2})=\alpha^{p^2+1}+\alpha^{p+1}+\alpha^{2p}+\alpha^{p^2+p}. 
\]
Thus, the difference $\Tr_{p^2/p}(\alpha^{p^2+1})-(a+1)^{p+1}$ is equal to 
\[\alpha^{p^3+p}+\alpha^{p+1}+\alpha^{2p}+\alpha^{p^2+p}=\alpha^p(\alpha+\alpha^p+\alpha^{p^2}+\alpha^{p^3}). \]
Since $\alpha+\alpha^p=a+1 \in \F_{p^2}$, we also have $\alpha^{p^2}+\alpha^{p^3}=a+1$. 
The assertion follows. 

(2) (ii) First we compute 
\[
\Tr_{p^4/p^2}(\a,0)=(\alpha,0)+(\alpha^{p^2},0)=(\alpha+\alpha^{p^2},\alpha^{p^2+1})=(a+a^{p},\alpha^{p^2+1}).
\]
Applying $\Tr_{p^2/p}$ and using $a+a^p\in\F_p$ and (i), we have 
\begin{align*}
\Tr_{p^4/p}(\alpha,0)&=\Tr_{p^2/p}(a+a^{p},\alpha^{p^2+1}) \\
&=(0,(a+a^p)^2+\Tr_{p^2/p}(\alpha^{p^2+1}))=(0,(a+a^p)^2+(a+1)^{p+1}).
\end{align*}
Taking $\Tr_{p/2}$, we find 
\[\Tr_{p^4/2}(\alpha,0)=(0,\Tr_{p/2}(a+a^p+(a+1)^{p+1}))=(0,\Tr_{p/2}(a^{p+1}+1)).\]
This proves the assertion, since $[\F_p:\F_2]=f_0$. 
\end{proof}

\begin{proposition}\label{2pp}
Let $S(x)=x^p+a_0 x$ with 
 $a_0 \in \mathbb{F}_{p^2}$. 
 \begin{itemize}
 \item[{\rm (1)}]  We have $H_S \subset \mathbb{F}_{p^4}^2$. 
 \item[{\rm (2)}] 
Assume $\Tr_{p/2}(a^{p+1}_0)=1$. 
Then the curve $\overline{C}_S$ is $\mathbb{F}_{p^4}$-maximal.  
 \end{itemize}
\end{proposition}
\begin{proof}
(1)
 Let $a \in \mathbb{F}_{p^2}=V_S$. 
 We consider the equation
 \[ b^p + b = a S(a) =
 a^{p+1}+a_0 a^2 \in \mathbb{F}_{p^2}. \] 
By Lemma~\ref{xip4} (1), it follows that $b \in \mathbb{F}_{p^4}$. 
Thus, we have $H_S \subset \mathbb{F}_{p^4}^2$. 

(2) Recall that we fix a faithful character $\xi_2 \in W_2(\F_2)^{\vee}$ and write $\xi_p=\xi_2\circ\Tr_{p/2}$. 
By the inclusion $H_S \subset \F_{p^4}^2$ and Lemma~\ref{Stone} (3), the geometric Frobenius $\mathrm{Fr}_{p^4}$ acts on 
$H^1(\overline{C}_{S,\F},\overline{\Q}_{\ell})$ as scalar multiplication by $\tau_{S,\xi_p,p^4}$. 
Thus it suffices to show $\tau_{S,\xi_p,p^4}=-p^2$.  

We take $\a \in \F_{p^4}$ such that 
$\a^p+\a=a_0+1$ by Lemma~\ref{xip4} (1). 
By Theorem~\ref{hc}, 
\[
\tau_{S,\xi_p,p^4}=\xi_{p^4}(\a^{p/2},0)^{-1} \cdot (-1-\sqrt{-1})^{4f_0}=\xi_{p^4}(\a,0)^{-1} \cdot (-1)^{f_0}p^2. 
\]
This equals $-p^2$ by Lemma~\ref{xip4} (2) (ii). 
\end{proof}
\begin{corollary}\label{abc}
Let $S(x)=x^p+a_0 x$ with 
 $a_0 \in \mathbb{F}_{p^2}$. 
Assume that $a_0 \in \mathbb{F}_p$
 and $\Tr_{p/2}(a_0)=1$. 
Then the curve $\overline{C}_S$ is 
$\mathbb{F}_{p^4}$-maximal. 
\end{corollary}
\begin{proof}
The assertion follows from 
Proposition~\ref{2pp} (2) since $\Tr_{p/2}(a^{p+1}_0)=\Tr_{p/2}(a^2_0)=\Tr_{p/2}(a_0)=1$. 
\end{proof}

We construct finite quotients of $\overline{C}_S$ with $a_0=1$, which provide examples of maximal curves. 
\begin{corollary}\label{abcd}
Let $a \in (\Ker \Tr_{p/2})\setminus \{0\}$. 
There exists an $\F_2$-linearized polynomial 
$Q_a(x) \in \F_p[x]$ such that 
\[
a^{-2}\sum_{i=1}^{f_0}(a x)^{2^i}=Q_a(x^2+a x). 
\]
Assume that $f_0$ is odd. Then 
the smooth compactification of the affine curve 
defined by $z^2+z=x Q_a(x)$
is $\F_{p^4}$-maximal. 
\end{corollary}
\begin{proof}
For an additive polynomial $Q(x) \in \F_p[x]$, let $D_Q$ denote the smooth affine curve over $\F_p$ defined by $z^2+z=xQ(x)$. 

Let $S(x) = x^p + x$.  
Since $f_0$ is odd, we have $\Tr_{p/2}(1) = f_0=1$.  
Thus, by Corollary~\ref{abc}, the curve $C_S$ is $\mathbb{F}_{p^4}$-maximal.  
The map $(x, y) \mapsto \bigl(x, \sum_{i = 0}^{f_0 - 1} y^{2^i} \bigr)$ defines a finite \'etale map $C_S \to D_S$.  

We apply Lemma~\ref{phif} to obtain a quotient of $D_S$.  
Let $a \in (\Ker \Tr_{p/2}) \setminus \{0\}$.  
Define the $\F_2$-linearized polynomial 
$
P(x) = a^{-2} \sum_{i = 1}^{f_0}(a x)^{2^i} \in \F_p[x].
$
Since $a \in \Ker \Tr_{p/2}$, we have $P(a) =a^{-2} \Tr_{p/2}(a)^4 =0$, and hence we can write $P(x) = Q_a(x^2 + a x)$ for an $\F_2$-linearized polynomial 
$Q_a(x) \in \mathbb{F}_p[x]$. 

Let 
\[
\Delta_0(x) \coloneqq a^{-2} \sum_{i = 0}^{f_0 - 1}(a x)^{2^i + 1} \in \F_p[x]. 
\]
Then, we can directly verify the identities 
\begin{align*}
\Delta_0(x)^2 + \Delta_0(x) 
&=x^{p+1}+x^2+(x^2+a x) P(x) \\
&= x^{p + 1} + x^2 + (x^2 + a x) Q_a(x^2 + a x).
\end{align*}
Consequently, by Lemma~\ref{phif}, we obtain a finite \'etale map $D_S \to D_{Q_a}$, which shows that $D_{Q_a}$ is a finite quotient of $C_S$.  
Since $C_S$ is $\mathbb{F}_{p^4}$-maximal by Corollary~\ref{abc}, the assertion follows.
\end{proof}

\section{The $L$-polynomials of twists of van der Geer--van der Vlugt curves}
\label{Section:Twists}

Let $R(x)\in\mathscr{A}$ be an $\F_p$-linearized polynomial of degree $p^e$ with $e\geq1$, with associated curve $C_R$. Van der Geer--van der Vlugt curves defined by $R(x)+ax$ for some $a\in \F_q$ are called \emph{twists} of $C_R$. 
In this section, we study such twists as applications of the results from previous sections. 

\subsection{Twists of van der Geer--van der Vlugt curves}
We first note that twists share the same polynomials defined in Section \ref{Section:Review}.  
\begin{lemma}\label{222}
For any $a \in \F_q$, we have $E_R(x)=E_{R+ax}(x)$, 
$f_R(x,y)=f_{R+ax}(x,y)$, and $\omega_R(x,y)=\omega_{R+ax}(x,y)$. In particular, we have $V_R=V_{R+ax}$, and any totally isotropic subspace
$U \subset V_R$ with respect to $\omega_R$ is
also totally isotropic with respect to $\omega_{R+ax}$.
\end{lemma}
\begin{proof}
From \eqref{-aa}, we have $E_R(x)=E_{R+ax}(x)$. 
Using equation~\eqref{-a}, we compute:
\[
f_{R+ax}(x,y) = \sum_{i=0}^{e-1}
\sum_{j=0}^{e-i-1}(a_i^{p^i}x^{p^i} y)^{p^j}+
\sum_{j=0}^{e-1}(a x y)^{p^j}+\sum_{i=0}^{e-1}(x R(y)+a x y)^{p^i} = f_R(x,y).
\]
Since 
$\omega_{R}(x,y)=f_R(x,y)+f_R(y,x)$, we obtain  $\omega_R(x,y)=\omega_{R+ax}(x,y)$. 
\end{proof}

For an $\F_p$-linearized polynomial $g(x)=\sum_{i=0}^mb_ix^{p^i} \in \mathscr{A}$, we set 
\[g^\ast(x)\coloneqq \sum_{i=0}^m(b_ix)^{p^{-i}}\in\F_q\bigl[x^{p^{-\infty}}\bigr].\]
Then we have $g(h(x))^\ast=h^\ast(g^\ast(x))$ for $g(x), h(x)\in\mathscr{A}$. 
\begin{lemma}\label{adjoint}
For two polynomials $f_1(x), f_2(x) \in \F_q[x]$, we write $f_1(x) \sim f_2(x)$ if there exists a polynomial $d(x) \in \F_q[x]$ such that
\[
d(x)^p + d(x) = f_1(x) + f_2(x).
\]
For $f(x) \in \mathscr{A}$ and $t \in \F_q$, we have
\[
f^\ast(t) x \sim t f(x).
\]
\end{lemma}

\begin{proof}
We write $f(x)=\sum_{i=0}^m b_i x^{p^i}$.
Since $(ax)^{p^i} \sim ax$ for any integer $i \ge 0$ and any $a \in \F_q$, we obtain 
\[
f^\ast(t) x
= \sum_{i=0}^m (b_i t)^{p^{-i}} x
\sim \sum_{i=0}^m b_i t x^{p^i}
= t f(x).
\]
This proves the lemma.
\end{proof}

We follow the notation and assumptions in Subsection \ref{R1R2} up to Proposition~\ref{221}. Let $\overline{A}_1\subset V_{R_1}$ be a maximal totally isotropic subspace containing $\Ker f$ and put $\overline{A}_2\coloneqq f(\overline{A}_1)$. Then 
$|\overline{A}_2|=|\overline{A}_1/\Ker f|=\deg R_1/\deg f=\deg R_2$ by Lemma~\ref{degcomp}. 
Hence Proposition~\ref{KerfA} (1) implies that 
$\overline{A}_2$ is a maximal totally isotropic subspace  of $V_{R_2}$, and that the sequence 
\[0\to \Ker f\to\overline{A}_1\xrightarrow{f}\overline{A}_2\to0 \]
is exact. 
We fix a separable polynomial $F_2(x)\in\mathscr{A}$ such that $\Ker F_2=\overline{A}_2$. We define  
$F_1(x)\coloneqq F_2(f(x))$.
By the above exact sequence, we have $\Ker F_1=\overline{A}_1$. 
\begin{proposition}\label{221}
Let $t \in \F_q$ and 
let 
\[
R_{1,t}(x)\coloneqq R_1(x)+F_1^\ast(t)^2x, \qquad 
R_{2,t}(x)\coloneqq R_2(x)+F_2^\ast(t)^2 x.
\]
\begin{itemize}
\item[{\rm (1)}] 
There exists a unique polynomial $\Delta_t(x) \in \F_q[x]$ such that $\Delta_t(0)=0$ and 
\[
\Delta_t(x)^p+\Delta_t(x)=xR_{1,t}(x)+f(x)R_{2,t}(f(x)). 
\]
\item[{\rm (2)}] The morphism 
\[
\varphi_t \colon C_{R_{1,t}} \to C_{R_{2,t}}, \quad  
(x,y) \mapsto (f(x), y+\Delta_t(x)) 
\]
is a finite \'etale Galois covering 
with Galois group $\Ker f$. 
\end{itemize}
\end{proposition}
\begin{proof} 
(1) 
By Lemma~\ref{adjoint}, we have 
$F_1^\ast(t)x=f^\ast(F_2^\ast(t))x\sim F_2^\ast(t)f(x)$.
Therefore, we obtain 
\begin{align*}
0& \sim x R_1(x)+f(x)R_2(f(x)) \\
&=xR_{1,t}(x)+(F_1^\ast(t)x)^2+f(x)R_2(f(x)) \\
& \sim xR_{1,t}(x)+(F_2^\ast(t) f(x))^2+f(x) R_2(f(x)) \\
&=xR_{1,t}(x)+f(x) R_{2,t}(f(x)). 
\end{align*}

(2) It follows from (1) and Lemma~\ref{phif}. 
\end{proof}

From now on, we consider the situation in Proposition~\ref{choice}. 
Put 
\[
F_2(x)\coloneqq x^p+x, \qquad F(x) \coloneqq  F_2(f_{\oU}(x))=f_{\oU}(x)^p+f_{\oU}(x).
\]
Then $\Ker F_2=\F_p$ is maximal totally isotropic with respect to $\omega_S$ in $V_S=\F_{p^2}$ and $\overline{A}=\Ker F$ by Proposition~\ref{choice}. 
Note that $F_2^\ast(x)=x^{p^{-1}}+x$. 
For $t\in \F_q$, we set 
\[R_t(x) \coloneqq  R(x)+F^\ast(t)^2x,
\qquad S_t(x) \coloneqq  S(x)+(t^{p^{-1}}+t)^2x.\]
By applying Proposition~\ref{221} (1) to $(R,S,f,F_1)=(R_1,R_2,f_{\oU},F)$, there exists a polynomial $\Delta_t(x) \in \F_q[x]$ such that 
\begin{equation}\label{delaa}
\Delta_t(x)^p+\Delta_t(x)=x R_t(x)+f_{\oU}(x) S_t(f_{\oU}(x)). 
\end{equation}
\begin{lemma}\label{atc}
Let
\[
\pi_t \colon H_{R_t}\to V_{R_t}=V_R,\quad (a,b)\mapsto a,
\]
and
\[
A_t \coloneqq \pi_t^{-1}(\overline{A}).
\]
Then $A_t$ is a maximal abelian subgroup of $H_{R_t}$ with
$A_t \subset \F_q^2$.
\end{lemma}
\begin{proof} 
By Lemma~\ref{222}, $\overline{A}$
is maximal totally isotropic with respect to 
$\omega_{R_t}$. 
Hence $A_t$ is a maximal abelian subgroup of $H_{R_t}$.  

We show $A_t \subset \F_q^2$.   
Let $a \in \overline{A} \subset \F_q$. 
By the assumption $\pi^{-1}(\overline{A}) \subset \F_q^2$ and Lemma~\ref{Hilbert90}, it follows that 
$\Tr_{q/p}(aR(a))=0$. 
Thus, by Lemma~\ref{adjoint}, we obtain
\[
\Tr_{q/p}(a R_t(a))
=\Tr_{q/p}\big(aR(a)+(F^\ast(t)a)^2\big)
=\Tr_{q/p}\big(tF(a)\big)^2=0,
\]
where the last equality follows from $\overline{A}=\Ker F$.
By Lemma~\ref{Hilbert90}, the claim follows. 
\end{proof}
We may apply Proposition~\ref{choice} to $(R,S,A,\Delta)=(R_t,S_t,A_t,\Delta_t)$ by \eqref{delaa} and Lemma~\ref{atc}.

Let $\xi_{W_2} \in W_2(\F_p)^{\vee}$ be a faithful character and let $\xi_{\oU} \in \oU^{\vee}$. Choose $\alpha\in \F_q$ as in Lemma~\ref{kk} (1) and $\beta_{\xi}\in\F_q$ as in \eqref{bxi}. 
Observe that 
$\oU=\{a \in \overline{A} \mid 
f_{R_t}(a,a)=0\}$ by Lemma~\ref{222}. By \eqref{exRS}, we have 
a commutative diagram of exact sequences 
\[
    \xymatrix{
0\ar[r]&\oU\ar[d]^-{{\rm id}_{\oU}}\ar[r]&A_t\ar[d]^-{\pi_t}\ar[r]& A_{S_t}\ar[d]^-{\pi_{S_t}}\ar[r]&0\\
0\ar[r]&\oU\ar[r]^-{i}&\overline{A}\ar[r]^-{f_{\oU}}& \F_p\ar[r]&0. 
}
\] 
We take a homomorphism $r \colon \overline{A} \to \oU$ such that $r \circ i=\id_{\oU}$
as in \eqref{retAU}. 
We write $S(x)=x^p+s_0 x$. Since $S_t(x)=x^p+((t^{p^{-1}}+t)^2+s_0) x$ and $\a^{p^{-1}}+\a=s_0+1$ as in Lemma~\ref{kk} (1), we have 
\[
(t^2+\a)^{p^{-1}}+(t^2+\a)=(t^{p^{-1}}+t)^2+s_0+1.
\]
Hence, as in \eqref{AUW}, the homomorphisms
$r \circ \pi_t \colon A_t \to \oU$ and
$f_{(t^2+\a)^{p^{-1}}}$ (see \eqref{aw}) 
together induce an isomorphism
\[
A_t \xrightarrow{\sim} W_2(\F_p) \times \oU.
\]
 We regard $\xi_{W_2} \boxtimes \xi_{\oU}$
as a character of $A_t$ via the isomorphism 
$A_t \simeq W_2(\F_p) \times 
\oU$, which we denote by $\xi_t$. 

The second main theorem of this paper, stated below, compares the Frobenius eigenvalues of $\overline{C}_{R_t}$ and $\overline{C}_{R_0}$.
\begin{theorem}\label{qppt}
\begin{itemize}
\item[{\rm (1)}] 
We have an isomorphism of $\overline{\Q}_{\ell}$-sheaves
\[ \mathcal{Q}_{\xi_t} \simeq 
\mathcal{L}_{\xi_{W_2}}(s,(t^2+\a) s^2+\beta_{\xi}s). \] 
\item[{\rm (2)}] 
Assume that 
$\xi_{W_2}=\xi_2 \circ \Tr_{(a_{W_2},b_{W_2})}$,
where $(a_{W_2},b_{W_2}) \in W_2(\F_p)$ and $a_{W_2} \neq 0$.
For a positive integer $n$ such that 
$\F_q \subset \F_{2^n}$, we have 
\[
\tau_{R_t,\xi_t,2^n}=\xi_{2^n}(t,t^2+ct)\cdot \tau_{R_0,\xi_0,2^n},
\]
where $c \coloneqq (\a+(b_{W_2}/a_{W_2}^2))^{1/2}+a_{W_2} \beta_{\xi} \in \F_q$. 
\end{itemize}
\end{theorem}

\begin{proof}
(1) 
The assertion follows from Proposition~\ref{qpp} by taking $t+\a$ for $\a$.

(2) 
By Theorem~\ref{qppc}, we obtain 
\[
\tau_{R_t,\xi_t,2^n}=\xi_{2^n}(t+c,0)^{-1} \cdot (-1-\sqrt{-1})^n.
\]
Therefore, the assertion follows since 
$(c,0)-(t+c,0)=(t,t^2+ct)$. 
\end{proof}

\subsection{Maximal twists of van der Geer--van der Vlugt curves}

As an application of our results,
we show how to obtain a maximal van der Geer--van der Vlugt curve
by twisting a minimal one.

Let $R(x)\in\mathscr{A}$ be an $\F_p$-linearized polynomial of degree $p^e$ with $e\geq1$. 
Let $A$ be a maximal abelian subgroup of $H_R$ 
such that $A \subset \F_q^2$. 
Let $\overline{A}\coloneqq \pi(A)$ and 
$\oU\coloneqq \{x \in \overline{A} \mid f_R(x,x)=0\}$. 
We take a separable polynomial $f_{\oU}(x) \in \mathscr{A}$ as in Proposition~\ref{choice}. 
Let $F(x)\coloneqq f_{\oU}(x)^p+f_{\oU}(x)$. 
We write $F(x)=\sum_{i=0}^e b_i x^{p^i}$. 
For $t\in \F_q$, we set 
\[
R_t(x) \coloneqq R(x)+F^\ast (t)^2 x \quad 
\textrm{with $F^\ast(t)=\sum_{i=0}^e (b_i t)^{p^{-i}}$}.
\]

\begin{theorem}\label{main}
Let $t \in \F_q$ be an element such that $\Tr_{q/2}(t)=1$.
\begin{itemize}
\item[{\rm (1)}]
Assume that $A \subset \F_q^2$ and 
$\overline{C}_R$ is $\F_{q^2}$-minimal.  
Then $\overline{C}_{R_t}$ is $\F_{q^2}$-maximal. 
\item[{\rm (2)}] 
Assume $H_R \subset \F_q^2$. 
Then $\overline{C}_{R_t}$ is $\F_{q^2}$-maximal. 
\end{itemize}
\end{theorem}

\begin{proof}
(1) 
We apply Theorem~\ref{qppt} (2). 
Let $\xi_{W_2}\in W_2(\F_p)^\vee$ be any faithful character, and let $\xi_{\oU}\in \oU^\vee$ be any character.
 We write $\xi_t$ for the character of $A_t$ corresponding to $\xi_{W_2}\boxtimes \xi_{\oU}$. 
By Theorem~\ref{qppt} (2), 
\[\tau_{R_t,\xi_t,q^2}=\xi_{q^2}(t,t^2+ct)\cdot\tau_{R_{0},\xi_0,q^2}. 
\]
By the assumption and $R=R_0$, 
the eigenvalue $\tau_{R_{0},\xi_0,q^2}$ equals $q$.
Since $(t,t^2+ct)\in W_2(\F_q)$ and $\Tr_{q/2}(t)=1$, we have
\[
\xi_{q^2}(t,t^2+ct)
=\xi_q(2 (t,t^2+ct))
=\xi_q(0,t^2)
=(-1)^{\Tr_{q/2}(t)}=-1.
\]
This proves the claim.

(2) By \cite[Theorem~1.2 (3)]{TT}, 
if $H_R \subset \F_q^2$, the curve 
$\overline{C}_R$ is $\F_{q^2}$-minimal. 
Hence the claim follows from (1). 
\end{proof}
Since finite quotients of maximal curves are maximal,
Theorem~\ref{main} implies the following conclusion.
\begin{corollary}
Let the notation and assumptions be as in Theorem~\ref{main} (2). 
\begin{itemize}
\item[{\rm (1)}] 
Let $s(x) \in \F_p[x]$ be an $\F_2$-linearized polynomial 
dividing $x^p+x$. 
Then the smooth compactification  of the affine curve over 
$\F_q$ defined by $s(y)=x R_t(x)$ is $\F_{q^2}$-maximal.
\item[{\rm (2)}] 
For any subgroup $V \subset \oU$, the smooth compactification of the quotient $C_{R_t}/V$ 
is $\F_{q^2}$-maximal. 
\end{itemize}
\end{corollary}
\begin{proof}
(1) 
Let $D_s$ be the affine curve over 
$\F_q$ defined by $s(y)=x R_t(x)$. 
There exists an $\F_2$-linearized polynomial 
$s_1(x) \in \F_p[x]$ such that $x^p+x=s(s_1(x))$. 
This yields a finite \'etale morphism $C_{R_t} \to D_s,\ (x,y) \mapsto (x,s_1(y))$, which extends to a finite morphism $\overline{C}_{R_t}
\to \overline{D}_s$.
Therefore, the claim follows from Theorem~\ref{main} (2). 

(2) 
This follows from Theorem~\ref{main} (2).
\end{proof}

As a corollary of Theorem~\ref{main}, we give explicit examples of maximal curves.

\begin{corollary}
\label{maximalexample}
Let $n$ be a positive even integer and 
$q=p^{4n}$. 
Let $t \in \F_q$ be an element such that $\Tr_{q/2}(t)=1$.
We define 
\[
R_t(x) \coloneqq \sum_{i=1}^{n/2} x^{p^{2i-1}}+\left(\sum_{i=0}^{n-1} t^{p^{-i}}\right)^2 x.
\]
Then the curve $\overline{C}_{R_t}$ 
is $\F_{q^2}$-maximal.  
\end{corollary}
\begin{proof}
The claim follows from Lemma~\ref{3.3} below by the following argument.
We define 
\[
R(x) \coloneqq \sum_{i=1}^{n/2} x^{p^{2i-1}}, \qquad 
f_{\oU}(x) \coloneqq \sum_{i=0}^{(n/2)-1}
x^{p^{2i}}, \qquad 
F(x) \coloneqq \sum_{i=0}^{n-1} x^{p^i}.
\]
Then all the conclusions of Proposition~\ref{choice}
are satisfied by Lemma~\ref{3.3} (2)--(5). 
Moreover, we have $H_R \subset \F_q^2$ by Lemma~\ref{3.3} (1). 
Therefore, all the assumptions in Theorem~\ref{main} (2)
are satisfied. 
Since
$F^{\ast}(t) = \sum_{i=0}^{n-1} t^{p^{-i}}$ and 
$R_t(x)=R(x)+F^\ast(t)^2 x$, 
the claim follows from Theorem~\ref{main} (2).
\end{proof}

\begin{lemma}\label{3.3}
Let $n$ be a positive even integer and 
$q=p^{4n}$. 
Let $R(x)\coloneqq \sum_{i=1}^{n/2} x^{p^{2i-1}}$. 
\begin{itemize}
\item[{\rm (1)}] We have  
$E_R(x)=\sum_{i=0}^{n-1} x^{p^{2i}}$, $V_R=\Ker \Tr_{p^{2n}/p^2} \subset \F_{p^{2n}}$ and $H_R \subset \F_q^2$. 
\item[{\rm (2)}]
Let $F(x)\coloneqq \sum_{i=0}^{n-1} x^{p^i}$. 
The subspace 
\[
\overline{A}\coloneqq \Ker \Tr_{p^n/p}
=\Ker F
\subset V_R
\]
 is a maximal totally isotropic subspace with respect to 
$\omega_R$. 
\item[{\rm (3)}] 
Let 
$\oU\coloneqq \Ker \Tr_{p^n/p^2}$ and 
$f_{\oU}(x)\coloneqq \sum_{i=0}^{(n/2)-1}
x^{p^{2i}}$. 
Then we have 
$\oU=\Ker f_{\oU} \subset \overline{A}$ and 
$f_{\oU}(x)^p+f_{\oU}(x)=
F(x)$. 
\item[{\rm (4)}]
There exists a unique polynomial $\Delta(x) \in \F_q[x]$ such that $\Delta(0)=0$ and 
\[
\Delta(x)^p+\Delta(x)=x R(x)+f_{\oU}(x)^{p+1}. 
\]
\item[{\rm (5)}] 
We have $\oU=\Ker (\overline{A} \to 
\F_p,\ a \mapsto f_R(a,a))$. 
\item[{\rm (6)}]
Let $S(x) \coloneqq x^p$. 
Then the morphism 
\[
C_R \to C_S, \quad (x,y) \mapsto (f_{\oU}(x),y+\Delta(x)) 
\] 
is a finite \'etale Galois covering with Galois group $\oU$. 
\end{itemize}
\end{lemma}
\begin{proof}
(1) We have 
\[
E_R(x)=R(x)^{p^{n-1}}+\sum_{i=1}^{n/2} 
x^{p^{n-2i}}=\sum_{i=1}^{n/2} x^{p^{n+2i-2}}
+\sum_{i=1}^{n/2} 
x^{p^{n-2i}}=\sum_{i=0}^{n-1} x^{p^{2i}}, 
\] 
and hence $V_R=\Ker \Tr_{p^{2n}/p^2} \subset \F_{p^{2n}}$. 
Since 
$\F_{p^{2n}} \subset \Ker \Tr_{p^{4n}/p}
=\{x^p+x \mid x \in \F_{p^{4n}}\}$, it follows that $H_R \subset \F_q^2$.

(2) 
Let $e \coloneqq n-1$. 
Then 
\[
F(F(x))=\sum_{i=0}^e \sum_{j=0}^e x^{p^{i+j}} 
=\sum_{k=0}^{2e}(k+1) x^{p^k}=\sum_{i=0}^e x^{p^{2i}}=E_R(x). 
\]
According to Lemma~\ref{adjoint}, for any $t\in\F_q$,
there exists a polynomial $d_t(x)\in\F_q[x]$ with $d_t(0)=0$ such that
\[
d_t(x)^p+d_t(x)=tF(x)+F^\ast(t)x.
\]
Substituting 
$F(x)$ into $x$ and using the identity  $E_R(x)=F(F(x))$, we obtain
\[ d_t(F(x))^p+d_t(F(x))=t E_R(x)+F^\ast(t) F(x). \] 
Now let $a\in\overline A=\Ker F$, and substitute $t=a^{p^e}$ 
into the above identity. Then
\begin{equation}\label{de}
d_{a^{p^e}}(F(x))^p+d_{a^{p^e}}(F(x))=a^{p^e} E_R(x)
\end{equation}
since $F^\ast(a^{p^e})=F(a)=0$. 
From \eqref{omega} and $E_R(a)=0$, 
we see that 
\[
\omega_R(x,a)^p+\omega_R(x,a)=a^{p^e} E_R(x).
\]
Moreover, 
since $d_{\a^{p^e}}(0)=\omega_R(0,a)=0$, it follows that 
$\omega_R(x,a)=d_{a^{p^e}}(F(x))$.  
In particular, for any 
$a' \in \overline{A}=\Ker F$, we obtain   
\[
\omega_R(a',a)=d_{a^{p^e}}(F(a'))=d_{a^{p^e}}(0)=0.
\] 
This shows that $\overline{A}$ is totally isotropic.  
Since $|\overline{A}|=p^e$ and $|V_R|=p^{2e}$, the maximality of 
$\overline{A}$ follows. 

(3) All assertions are directly checked. 

(4) Let $\sim$ be as in Lemma \ref{adjoint}. 
The claim follows from the following calculation.
\begin{align*}
f_{\oU}(x)^{p+1}&=\sum_{0\le i,j\le (n/2)-1} x^{p^{2i+1}+p^{2j}} \\
&=\sum_{0\le i<j\le (n/2)-1} x^{p^{2i+1}(p^{2(j-i)-1}+1)}+\sum_{0\le j \le i\le (n/2)-1} x^{p^{2j}(p^{2(i-j)+1}+1)} \\
& \sim \sum_{0\le i<j\le (n/2)-1} x^{p^{2(j-i)-1}+1}+\sum_{0\le j \le i\le (n/2)-1} x^{p^{2(i-j)+1}+1}  \\
&=\sum_{k=1}^{(n/2)-1}\sum_{i=0}^{(n/2)-k-1} 
x^{p^{2k-1}+1}+\sum_{k=1}^{n/2}\sum_{j=0}^{(n/2)-k} x^{p^{2k-1}+1}=\sum_{k=1}^{n/2} x^{p^{2k-1}+1}=xR(x). 
\end{align*}

(5) 
From Lemma~\ref{Kerf} (2), applied to
$(R_1(x),R_2(x),f(x))=(R(x),x^p,f_{\oU}(x))$ and $\oU=\Ker f_{\oU}$,
it follows that
$\oU \subset 
\Ker (\overline{A} \to \F_p,\ a \mapsto f_R(a,a))$.
Since both sides have the same cardinality by Lemma \ref{suc}, the inclusion is an equality.

(6) 
It follows from (4), $\oU \subset \F_q$, and Lemma~\ref{phif} (2). 
\end{proof}

\section{Supersingular elliptic curves as quotients of van der Geer--van der Vlugt curves}
\label{Section:supersingular}

In this section, we show that supersingular elliptic curves naturally 
appear as quotients of the curve $C_S$ defined by 
\[
  y^p+y = x^{p+1}+a_0 x^2 \qquad (a_0 \in \F_q)
\]
and the defining equations of these elliptic 
curves can be written explicitly. 
Let $R(x) \in \mathscr{A}$ with $\deg R>1$. 
Using the finite \'etale morphism 
$C_R \to C_S$ in \eqref{rr}, 
we obtain finite \'etale morphisms  
from $C_R$ to supersingular elliptic curves.
In the case where $p=p_0$, 
similar results are proved in \cite[Section 9]{GV}.

We use the notation from Subsection \ref{LangTorsor}. Assume that there exists 
$\a \in \F_q$ such that 
$\a^{p^{-1}}+\a=a_0+1$ (cf.\ Lemma~\ref{lemmaalpha}). 
Let $(a,b) \in W_2(\F_p)$ with $a \neq 0$. 
Let $T_{\alpha, (a,b)} \subset W_{2,\F_q}$ denote the closed subscheme defined by the closed immersion 
\[ 
\A_{\F_q}^1 \hookrightarrow W_{2,\F_q}, \quad s \mapsto 
(a,b)\cdot (1,\alpha) \cdot (s,0)=(a s, (a^2\a+b) s^2).
\]
If $(a,b)=(1,0)$, then $T_{\a,(1,0)}$ coincides with 
$T_{\a}$ in \eqref{coo}. 

There exists a morphism 
\[
h_0 \colon L_p^{-1}(T_{\a}) \to L_2^{-1}(T_{\alpha, (a,b)}), \quad  (x,y) \mapsto \sum_{i=0}^{f_0-1}\left((ax)^{2^i}, (bx^2+a^2y)^{2^i}\right). 
\]
This morphism is well-defined because for $(x,y) \in L_p^{-1}(T_{\a})$, we have 
\begin{gather}\label{com1}
\begin{aligned}
L_2\left(\sum_{i=0}^{f_0-1}\left((ax)^{2^i}, (bx^2+a^2y)^{2^i}\right)\right)&=((ax)^p, (bx^2+a^2y)^p)
-(ax,bx^2+a^2 y) \\
&=((ax)^p, (bx^2+a^2y)^p)
+(ax,bx^2+a^2 y+(ax)^2) \\
&=(a(x^p+x), b(x^p+x)^2+a^2(y^p+y+x^{p+1}+x^2))
\\&=(a s, (a^2 \a+b) s^2) \in T_{\a,(a,b)}, 
\end{aligned}
\end{gather}
where we set $s \coloneqq x^p+x$ and
use $y^p+y+x^{p+1}+x^2=\a s^2$ in \eqref{coo} at the last 
equality. 
Let $\Tr_{(a,b)} \colon W_2(\F_p) \to W_2(\F_2)$
be as in \eqref{trab}. 
\begin{proposition}\label{dec}
We write $p = 2^{f_0}$. 
The morphism $h_0$ induces an isomorphism 
\[
h \colon L_p^{-1}(T_{\a})/\Ker \Tr_{(a,b)}\xrightarrow{\sim} L_2^{-1}(T_{\alpha, (a,b)}), \quad 
(x,y) \mapsto \sum_{i=0}^{f_0-1}\left((ax)^{2^i},(bx^2+a^2y)^{2^i}\right). 
\]
\end{proposition}

\begin{proof}
Since $a \ne 0$, we have $(a, b) \in W_2(\F_p)^{\times}$, as explained before Example~\ref{wex}. 
Hence, by Lemma~\ref{surjTr}, the map 
$\Tr_{(a,b)} \colon W_2(\F_p) \to 
W_2(\F_2)$ is surjective, 
and induces an isomorphism 
\[ W_2(\F_p)/\Ker \Tr_{(a,b)} \xrightarrow{\sim} W_2(\F_2). \] 
Recall that, for $(\alpha,\beta) \in W_2(\F_p)$, 
\[
\Tr_{(a,b)}(\alpha,\beta)
=\Tr_{p/2}((a,b)\cdot(\alpha,\beta))
=\sum_{i=0}^{f_0-1}\bigl((a \alpha)^{2^i},(b \alpha^2+a^2 \beta)^{2^i}\bigr).
\]
For $(x,y) \in L_p^{-1}(T_{\a})$ and $(\alpha,\beta) \in \Ker \Tr_{(a,b)}$, we have
\[
h_0((x,y)+(\alpha,\beta))
=h_0(x+\alpha, y+\beta+\a x)
= h_0(x,y)+\Tr_{(a,b)}(\alpha,\beta)
= h_0(x,y).
\]
Hence $h_0$ factors through the quotient morphism
\[
L_p^{-1}(T_{\a}) \to L_p^{-1}(T_{\a})/\Ker \Tr_{(a,b)}.
\]

By \eqref{com1}, we have the following commutative diagram 
\[
\xymatrix{
L_p^{-1}(T_{\a}) \ar@/^25pt/[rr]^-{h_0}\ar[r]\ar[dr]_-{\pi_0} & L_p^{-1}(T_{\a})/\Ker \Tr_{(a,b)} \ar[r]^-h\ar[d]^-{\pi_1} &
L_2^{-1}(T_{\alpha, (a,b)}) \ar[d]^-{\pi_2}\\
 & T_{\a} \ar[r]^-{h_1}_-{\simeq} & T_{\alpha, (a,b)},
}
\]
where $\pi_0, \pi_1, \pi_2$ denote the natural projections, and
$h_1$ is defined by
\[
(s,\alpha s^2) \mapsto (as, (a^2\alpha+b)s^2).
\]
Since $\pi_0$ is a finite \'etale Galois covering with 
Galois group $W_2(\F_p)$, the morphism $\pi_1$ is 
a finite \'etale Galois covering with 
Galois group $W_2(\F_p)/\Ker \Tr_{(a,b)} \simeq W_2(\F_2)$. 
Since $\pi_1$ and $\pi_2$ are finite \'etale of degree $4$, the finite morphism $h$ is an isomorphism.
\end{proof}

\begin{proposition}\label{subel}
Let $c \coloneqq (\alpha+(b/a^2))^{1/2}$. 
Then the subscheme 
$L_2^{-1}(T_{\alpha, (a,b)})$ is defined by the equation
\begin{equation}\label{eleq}
z^2+z = x^3 + \bigl(c^2+c+1\bigr)x^2, 
\end{equation}

which is an affine equation defining a supersingular elliptic curve (cf.\ Remark \ref{quartic}). 
\end{proposition}

\begin{proof}
We have $L_2(x,y)=(x^2+x, y^2+y+x^3+x^2)$ by \eqref{coo-1} for $p=2$.  
By the definition of $T_{\alpha,(a,b)}$, the subscheme 
$L_2^{-1}(T_{\alpha, (a,b)})$ is defined by 
\[
x^2+x = a s, \qquad 
y^2+y = x^3+x^2+(a^2 \alpha + b)s^2. 
\]
By eliminating $s$ using the first equation,
substituting it into the second, and setting
$z=y+cx^2$, we obtain the required equation.
\end{proof}

Let $R(x) \in \mathscr{A}$ with $\deg R>1$. 
Let $A \subset H_R$ be a maximal 
abelian subgroup such that $A \subset \F_q^2$. Let  
$\oU \coloneqq  \{x \in \pi(A) \mid 
f_R(x,x)=0\}$, where $\pi \colon H_R \to V_R,\ 
(a,b) \mapsto a$. 
\begin{corollary}\label{qel}
Let $E_{\a,a,b}$ denote the affine curve 
defined by \eqref{eleq}. 
There exists a finite \'etale Galois covering 
$C_R \to E_{\a,a,b}$ with Galois group
$(\Ker \Tr_{(a,b)}) \times \oU$. 
\end{corollary}
\begin{proof}
We have a sequence of morphisms 
\[
C_R\xrightarrow{\eqref{rr}} C_S\cong L_p^{-1}(T_\alpha)\xrightarrow{h_0}L_2^{-1}(T_{\alpha,(a,b)})=E_{\alpha,a,b}\xrightarrow{(x,y)\mapsto x}\A^1_{\F_q}, 
\]
where the identification $C_S\cong L_p^{-1}(T_\alpha)$ is given by Lemma~\ref{dc}. By construction, the composite map $C_R\to \A^1_{\F_q},\ (x,y) \mapsto a(f_{\oU}(x)^p+f_{\oU}(x))$ is an abelian covering with Galois group $A$ (cf.\ \eqref{phic} and Proposition~\ref{choice}). Thus the intermediate covering $C_R\to E_{\alpha,a,b}$ is also an abelian covering. Let $G$ denote its Galois group. It fits into a short exact sequence 
\[
0 \to \overline{U} \to G \to \Ker \Tr_{(a,b)} \to 0. 
\]
By Lemma~\ref{kk} (2), this sequence splits. The assertion follows. 
\end{proof}
Let $\xi_2 \in W_2(\F_2)^{\vee}$ be a faithful character and 
define $\xi_q \coloneqq \xi_2 \circ \Tr_{q/2} \in W_2(\F_q)^{\vee}$. 
We also set $\sqrt{-1} \coloneqq \xi_2(1,0)$.
If $H_R \subset \F_q^2$, the Frobenius eigenvalues of 
$H^1(\overline{C}_{R,\F},\overline{\Q}_{\ell})$ are all equal to a single Frobenius eigenvalue of a supersingular elliptic curve.
We obtain the same conclusion as in Corollary~\ref{qaq2} (1). 
\begin{corollary}
We write $E_{\a}$ for $E_{\a,1,0}$. 
Assume $H_R \subset \F_q^2$. Then 
$\mathrm{Fr}_q$ acts on 
$H^1(\overline{C}_{R,\F},\overline{\Q}_{\ell})$
as scalar multiplication by $\xi_q(\a,0)^{-1} \cdot (-1-\sqrt{-1})^{[\F_q:\F_2]}$. 
In particular, $\overline{C}_R$ is a supersingular curve. 
\end{corollary}
\begin{proof}
We have the finite \'etale morphism
$
C_R \to E_{\a}
$
by  
Corollary~\ref{qel} for $(a,b)=(1,0)$. 
This morphism extends to a finite morphism 
$
\overline{C}_R \to \overline{E}_{\a},
$
which induces an injection 
\[
H^1(\overline{E}_{\a,\F},\overline{\Q}_{\ell}) 
\hookrightarrow H^1(\overline{C}_{R,\F},\overline{\Q}_{\ell}). 
\]
Since $H_R \subset \F_q^2$ and 
Lemma~\ref{Stone} (3), the Frobenius $\mathrm{Fr}_q$ acts on $H^1(\overline{C}_{R,\F},\overline{\Q}_{\ell})$ as scalar multiplication. 
By Remark \ref{quartic}, the scalar 
equals 
\[
\xi_q(c,0)^{-1} \cdot G_{\xi_q}=\xi_q(\a,0)^{-1} \cdot (-1-\sqrt{-1})^{[\F_q:\F_2]}
\]
since $c=\a^{1/2}$ and by Lemma~\ref{HDGr}. 
Since this scalar is 
a root of unity times $\sqrt{q}$, the last claim follows.  
\end{proof}
\subsection*{Acknowledgements}
The authors thank the referee for a careful reading of the manuscript and for many helpful and constructive suggestions.\\
T. I. is supported by JSPS KAKENHI Grant Numbers 23K20786, 24K21512 and 25K00905.\\
D. T. is supported by JSPS KAKENHI Grant Number 25KJ0122.\\
T. T. is supported by JSPS KAKENHI Grant Numbers 25K06959 and 23K20786.

\end{document}